\documentclass[reqno, 12pt]{amsart}
\usepackage{amsthm, amstext, amsmath, amssymb, graphicx, verbatim, fullpage}
\usepackage{booktabs}
\usepackage{hyperref}
\usepackage{ mathrsfs }
\usepackage{algorithm}
\usepackage{versions}
\usepackage{enumerate}
\usepackage{graphicx}
\usepackage{mathtools}
\usepackage{algorithm}
\usepackage{tabularx}
\usepackage{tikz-cd}
\usepackage{fullpage, microtype, colonequals}

\usepackage[utf8]{inputenc}
\usepackage[english]{babel}
\usepackage{xcolor}

\newtheorem{thm}{Theorem}[section]
\newtheorem{lemma}[thm]{Lemma} \newtheorem{cor}[thm]{Corollary}
\newtheorem{prop}[thm]{Proposition}
\theoremstyle{definition}
\newtheorem{defn}[thm]{Definition}

\newtheorem{conj}[thm]{Conjecture}
\newtheorem{example}[thm]{Example}

\newtheorem{remark}[thm]{Remark}

\providecommand{\new}[1]{#1}

\newcommand{\A}{\mathbb{A}}
\newcommand{\Z}{\mathbb Z}
\newcommand{\N}{\mathbb N}
\newcommand{\R}{\mathbb R}
\newcommand{\C}{\mathbb C}
\newcommand{\Q}{\mathbb Q}
\newcommand{\F}{\mathbb F}
\newcommand{\abfam}{\mathcal{A}}

\newcommand{\shO}{\mathcal{O}}

\newcommand{\PP}{\mathbb{P}}

\newcommand{\Tw}{\mathcal{T}_n}
\newcommand{\G}{\mathbb{G}}

\newcommand{\GITQ}{/ \! /}

\DeclareMathOperator{\diag}{diag}
\DeclareMathOperator{\charac}{char}

\DeclareMathOperator{\Id}{Id}
\DeclareMathOperator{\Spec}{Spec}
\DeclareMathOperator{\PGL}{PGL}
\DeclareMathOperator{\GL}{GL}
\DeclareMathOperator{\SL}{SL}

\DeclareMathOperator{\Mat}{Mat}

\DeclareMathOperator{\Dom}{Dom}

\DeclareMathOperator{\marked}{mark}
\DeclareMathOperator{\Pic}{Pic}
\DeclareMathOperator{\DynDom}{DynDom}

\DeclareMathOperator{\Weil}{Weil}

\DeclareMathOperator{\len}{len}
\DeclareMathOperator{\ord}{ord}

\newcommand*\subtxt[1]{_{\textnormal{#1}}}

\newcommand{\TwParamNice}{\mathcal{U}_n}

\newcommand{\LeftSpace}{\mathcal{V}_n}
\newcommand{\RightSpace}{\mathcal{W}_n}
\newcommand{\Dualize}{\Delta}
\newcommand{\DualizeBack}{\Dualize'}

\newcommand{\acLeft}{\alpha}
\newcommand{\acRight}{\beta}
\newcommand{\Shift}{\Sigma}

\newcommand{\Dplus}{D^{\scriptscriptstyle+}}
\newcommand{\Dminus}{D^{\scriptscriptstyle-}}
\newcommand{\tildeDplus}{\tilde{D}^{\scriptscriptstyle+}}
\newcommand{\tildeDminus}{\tilde{D}^{\scriptscriptstyle-}}

\newcommand{\Gm}{\G\subtxt{m}}

\newcommand{\boxnodd}{\boxed{\text{Odd } n}}
\newcommand{\boxneven}{\boxed{\text{Even } n}}

\newcommand{\Ind}{\mathcal{I}}

\newcommand{\nsrhmap}{\zeta}
\newcommand{\arithgenus}{{g\subtxt{a}}}

\DeclarePairedDelimiter\floor{\lfloor}{\rfloor}

\begin{document}

\title[Algebraic Dynamics of the Pentagram Map]{The Algebraic Dynamics of the Pentagram Map}
\author[M. H. Weinreich]{Max H. Weinreich}
\date{\today}
\email{maxhweinreich@gmail.com}
\address{Department of Mathematics, Brown University, Providence, RI 02906.
 ORCID: 0000-0002-0103-2245}
\thanks{The author was supported by a National Science Foundation Graduate Research Fellowship under Grant No. 2040433.}

\maketitle

\begin{abstract}
The pentagram map, introduced by Schwartz in 1992, is a dynamical system on the moduli space of polygons in the projective plane. Its real and complex dynamics have been explored in detail. We study the pentagram map over an arbitrary algebraically closed field of characteristic not equal to 2. We prove that the pentagram map on twisted polygons is a discrete integrable system, in the sense of algebraic complete integrability: the pentagram map is birational to a self-map of a family of abelian varieties. This generalizes Soloviev's proof of complex integrability. In the course of the proof, we construct the moduli space of twisted $n$-gons, derive formulas for the pentagram map, and calculate the Lax representation by characteristic-independent methods.
\end{abstract}

\section{Introduction}
\label{sect_intro_penta}

\subsection{Main result} The pentagram map is a discrete dynamical system on the space of polygons in the projective plane. The map was introduced by Schwartz in 1992 for convex polygons in the real projective plane \cite{MR1181089}, but the definition extends to polygons in any projective plane. This paper describes the dynamics of the pentagram map in projective planes over algebraically closed fields, including positive characteristic. Our main result establishes algebro-geometric complete integrability of the pentagram map over any arbitrary algebraically closed field of characteristic not equal to 2.

\new{
\begin{defn} \label{def_ngon}
Let $n \geq 3$ be an integer. A \emph{closed $n$-gon}, or just \emph{$n$-gon}, is an ordered $n$-tuple of points $(v_1, \hdots, v_n) \in (\PP^2)^n$ in general linear position. The space of $n$-gons is a Zariski open subset of $(\PP^2)^n$.
\end{defn}
}
\new{
\begin{defn} \label{def_early_penta}
Let $n \geq 5.$ The \emph{pentagram map} is a rational self-map of the space of $n$-gons. The pentagram map sends an $n$-gon $(v_1, \hdots, v_n)$ to the $n$-gon $(w_1, \hdots, w_n)$, where $w_i$ is the intersection of the diagonals $\overline{v_{i-1} v_{i+1}}$ and $\overline{v_i v_{i+1}}$, and where we take the indices cyclically modulo $n$; see Figure \ref{fig_nonagon}.
\end{defn}}

\begin{figure}[b]
\begin{center}
\begin{tikzpicture}
\draw[very thick] (2, 0) -- (3.4, 0.6) -- (4, 1.8) -- (3.6,2.9) -- (2.5,3.5) -- (1.3, 3.4) -- (0.4, 2.7) -- (0, 1.6) -- (0.6, 0.5) -- cycle;
\draw[->] (4.5,1.75) -- (5,1.75);
\draw[thin]
(7.5, 0) -- (8.9, 0.6) -- (9.5, 1.8) -- (9.1,2.9) -- (8,3.5) -- (6.8, 3.4) -- (5.9, 2.7) -- (5.5, 1.6) -- (6.1, 0.5) -- cycle;
\draw[thin]
(7.5, 0) -- (9.5, 1.8)  -- (8,3.5)  -- (5.9, 2.7) -- (6.1, 0.5) -- (8.9, 0.6)-- (9.1,2.9) -- (6.8, 3.4) -- (5.5, 1.6) -- cycle;
\draw[very thick]
(8.136,0.572) -- (8.962, 1.316) -- (9.05,2.31) -- (8.39,3.05) -- (7.4, 3.27) -- (6.44,2.91) -- (5.944,2.215) -- (6.039, 1.168) -- (6.842, 0.526) -- cycle;
\end{tikzpicture}
\end{center}
\caption{The pentagram map applied to a $9$-gon.}
\label{fig_nonagon}
\end{figure}
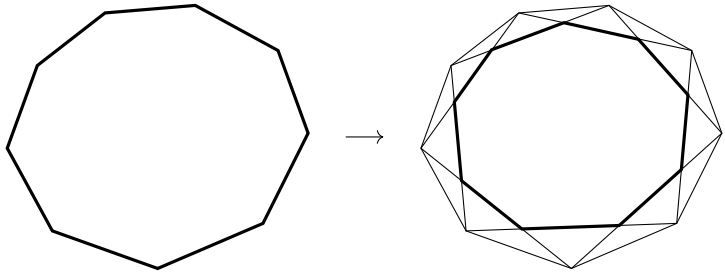

Since the construction is projectively natural, we get an induced rational self-map of the moduli space of $n$-gons in $\PP^2$ up to projective equivalence. From now on, by the \emph{pentagram map}, we mean this map on moduli space.

Schwartz, following computational evidence, conjectured that the real pentagram map might be a rare example of a \emph{Liouville-Arnold discrete integrable system}. This was proved in 2010-11 by Ovsienko, Schwartz, and Tabachnikov \cite{MR2679816, MR3102478}, leading to an explosion of work on the pentagram map, including higher-dimensional generalizations \cite{MR4430020, MR3118623, MR3282373} and connections to cluster algebras \cite{MR3534837, MR2793031, MR3282370}, Poisson-Lie groups \cite{MR3675462, MR4430020}, and integrable partial differential equations \cite{MR3041627}.

Liouville-Arnold integrability is an extremely strong property which almost completely describes the dynamics. Roughly, it means:
\begin{itemize}
    \item the domain of the map (dimension $\approx 2n$) admits a fibration by invariant submanifolds of dimension $\approx n$;
    \item each of these submanifolds may be identified with an open subset of a real torus of dimension $\approx n$, such that on each torus, some iterate of the pentagram map is a translation.
\end{itemize}

In this paper, we study the closely related \emph{pentagram map on twisted polygons}, also introduced by Schwartz \cite{MR2434454}. The main integrability theorems for the pentagram map on closed polygons were proved first for twisted polygons \cite{MR3102478, MR3161305}.

\begin{defn}
\new{
A \emph{twisted $n$-gon} is a sequence $(v_i)_{i \in \Z}$ in $\PP^2$ with the property that there exists a projective transformation $M \in \PGL_3$, called the \emph{monodromy}, such that, for all $i \in \Z$,
\begin{equation} \label{eq_def_twisted}
    Mv_i = v_{i + n}.
\end{equation}
We also impose some nondegeneracy conditions; see Definition \ref{def_twisted_ngon} for details. The \emph{pentagram map on the parameter space of twisted $n$-gons} sends $(v_i)_{i \in \Z}$ to $(w_i)_{i \in \Z}$, where $w_i$ is the intersection of the diagonals $\overline{v_{i-1} v_{i+1}}$ and $\overline{v_i v_{i+1}}$; see Figure \ref{fig_twisted_pentagram}.
For any $T \in \PGL_3$ and twisted $n$-gon $(v_i)_{i \in \Z}$, the twisted $n$-gons $(v_i)_{i \in \Z}$ and $(Tv_i)_{i \in \Z}$ are \emph{projectively equivalent}. The \emph{moduli space of twisted polygons}, denoted $\Tw$, is the quotient space of twisted polygons up to projective equivalence; we construct $\Tw$ as a variety in Section \ref{sect_moduli}. Since the pentagram map on the parameter space of twisted polygons is projectively natural, it descends to a rational self-map
$$f \colon \Tw \dashrightarrow \Tw.$$
}

\new{
From now on, by the \emph{pentagram map}, we mean this map $f$ on the moduli space.
}
\end{defn}

\new{To motivate the definition, notice that the definition of the pentagram map on closed polygons is combinatorially local, in the sense that each vertex of the image polygon $v$ depends only on four consecutive vertices of $v$. Thus the pentagram map extends to a self-map of the space of sequences $(\PP^2)^\Z$, and $n$-gons correspond to $n$-periodic sequences. Studying the pentagram map in this larger space is difficult because $(\PP^2)^\Z$ is infinite-dimensional. The constraint \eqref{eq_def_twisted} defines a finite-dimensional domain for the pentagram map. Closed $n$-gons are twisted $n$-gons that have monodromy $M = 1$. We think of closedness as a global constraint on the geometry of a twisted polygon.}

Our main theorem is an algebro-geometric version of discrete integrability which holds in characteristic $0$ and characteristic $p$.

\begin{thm} \label{thm_main_1}
   Let $k$ be an algebraically closed field.
   \begin{enumerate}
   \item The moduli space $\Tw$ of twisted $n$-gons over $k$ exists as an algebraic variety, and $\Tw$ is a rational variety of dimension $2n$.
   \item  Assume that $\charac k \neq 2$. Then there exists a family of abelian varieties
   $$\mathcal{A} \to S$$ and a birational map
   $$\delta \colon \Tw \dashrightarrow \mathcal{A},$$
   such that, via the identification $\delta$, the fibers of $\mathcal{A} \to S$ are invariant subvarieties for the pentagram map.
   \item The behavior of the pentagram map on $\mathcal{A}$ depends on the parity of $n$:
      \begin{itemize}
       \item \boxnodd: The fibers of $\mathcal{A}$ 
       are Jacobian varieties of dimension $n - 1$, and $\delta$ identifies the pentagram map with a translation by a section of $\mathcal{A} \to S$.
       \item \boxneven: The fibers of $\mathcal{A}$ are pairs of Jacobian varieties of dimension $n - 2$. Via the identification $\delta$, the pentagram map sends each Jacobian isomorphically to the other in its pair. The map $\delta$ identifies the second iterate of the pentagram map with a translation by a section of $\mathcal{A} \to S$.
   \end{itemize}
   \end{enumerate}
\end{thm}

By a \emph{family of abelian varieties} over $k$, we mean a map of $k$-schemes $A \to S$ such that each fiber is isomorphic to an abelian variety. We do not assume any choice of zero-section $S \to A$.

Theorem \ref{thm_main_1} can in fact be made totally explicit. Our methods give equations for the invariant subvarieties and the section that corresponds to the pentagram map.

The case of Theorem \ref{thm_main_1} with base field $k = \C$ is essentially due to Soloviev \cite{MR3161305}. We extend the result to algebraically closed fields of any characteristic except~$2$. 
\new{In fact, we expect that, with some additional work, our proof would extend to characteristic~$2$; see Remark \ref{rem_source_of_polys}.}

\new{
\begin{remark}A natural next step would be to formulate a notion of algebraic complete integrability over $\Spec \Z$. To explain this informally, while we study the pentagram map over each field independently, it is also true that any algebraic dynamical system $f$ defined only using integers can be thought of as a self-map of a some scheme over $\Spec \Z$. The scheme $\Spec \Z$ is a 1-dimensional topological space, and each prime number corresponds to a point of $\Spec \Z$. The codimension-1 fiber at prime $p$ corresponds to the dynamical system induced by $f$ over $\F_p$, and these fibers are all $f$-invariant. Of course, integrable systems have a more refined invariant fibration than this, usually including some ``degenerate'' leaves with interesting but non-generic dynamics. We propose that the bad primes for an integrable system over $\Z$, if any, should be thought of as degenerate leaves. Thus, ``generic'' properties of an integrable system over $\Z$ should hold at all but finitely many primes.
\end{remark}}

Theorem \ref{thm_main_1} has strong consequences for the arithmetic dynamics of the pentagram map over finite fields. For instance, the orbits of the pentagram map over a finite field $\F_q$ are much smaller than one would expect for a randomly chosen rational self-map of $\PP^{2n}$, thanks to standard estimates for point counts on varieties over finite fields. \new{Over $\F_q$, at least when $q$ is odd, the domain $\mathcal{T}_n$ of the map has $O(q^{2n})$ elements, but these can be divided into invariant subsets of cardinality $O(q^{n-1})$ or $O(q^{n -2})$, depending on the parity of $n$. On sufficiently generic invariant subsets, the well-defined orbits of the pentagram map within an invariant subset all have the same period. Note that well-definedness of orbits is an issue because} the pentagram map is a rational map rather than a morphism, since some degenerate polygons do not have well-defined images. A heuristic argument suggests that almost all orbits of the pentagram map eventually produce degenerate polygons. We formalize this idea in Conjecture \ref{conj_small_domain}.

Theorem \ref{thm_main_1} also tells us something about the real pentagram map. When $n$ is even, a typical nonperiodic orbit of a twisted $n$-gon fills out at least 2 tori, by taking real parts of $\abfam$. We show an example in Figure \ref{fig_fourgon}.

\begin{figure}[h]
\begin{center}
\includegraphics[width=2.4in]{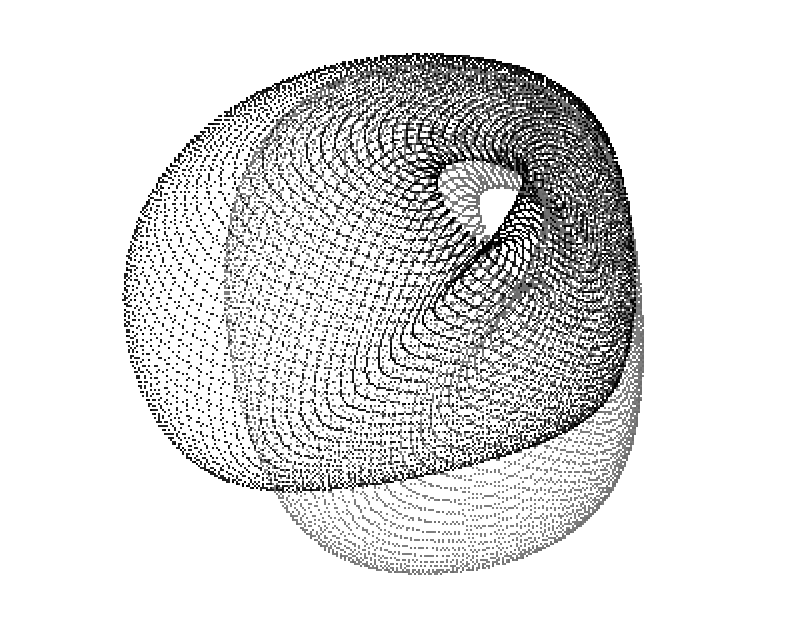}
\end{center}
\caption{The first few thousand iterates of a real twisted $4$-gon, depicted via a 2-dimensional projection from the 8-dimensional moduli space $\mathcal{T}_4$. Odd-indexed and even-indexed iterates alternate between two real 2-dimensional tori.}
\label{fig_fourgon}
\end{figure}

\new{A final application of Theorem \ref{thm_main_1} is to the pentagram map over $\Q$. The \emph{logarithmic height} of a rational number is a measure of its arithmetic complexity, approximately equal to the number of digits needed to write down the number in lowest terms; see \cite{ads}. One can extend this definition to rational points in projective space and ask about the height growth along orbits of a dynamical system. Conjecturally, polynomial growth of height in orbits is an indicator of integrability, but few theoretical results in this direction are known \cite{MR2131425}. In an experimental study, Khesin and Soloviev observed polynomial height growth in orbits of some generalized pentagram maps, and exponential height growth in others, providing heuristic evidence of integrability vs. non-integrability \cite{MR3282373}. We confirm this empirical result for the standard pentagram map: for sufficiently generic orbits, the logarithmic height growth is not just polynomial, but in fact linear (Corollary \ref{cor_height_growth}).}

\begin{remark}
The methods of this paper should also apply to closed polygons, with additional work. In that case, we expect an analogue of Theorem \ref{thm_main_1} to hold with the dimensions $n - 1$  and $n - 2$ replaced by $n - 4$ and $n - 5$. This was shown over $\C$ by Soloviev \cite{MR3161305}.
\end{remark}

\subsection{Sketch of proof of Theorem \ref{thm_main_1}}
In the complex theory of integrability, both algebraic and analytic methods are available. Working in characteristic $p$, we are restricted to algebraic techniques. We follow the route of Lax representations and spectral curves. Even though these techniques are classified as ``analytic'' in some textbooks, e.g. \cite{MR1995460}, they can be adapted to other algebraically closed fields.

The first step is to construct the moduli space $\Tw$ of twisted $n$-gons as an algebraic variety (Theorem \ref{thm_moduli}). The group $\PGL_3$ of projective transformations $M$ has dimension 8, so the parameter space of twisted $n$-gons has dimension $2n + 8$. Since the pentagram map commutes with $\PGL_3$, it descends to the moduli space $\Tw$ of twisted $n$-gons up to projective equivalence. The variety $\Tw$ is $2n$-dimensional, and the pentagram map is a rational self-map $f$. \new{ Since the projective equivalence class of $M$ is $f$-invariant, we have at least two algebraically independent integrals. }

The main technique we use to construct the moduli space $\Tw$ is geometric invariant theory (GIT), which supplies tools for taking quotients of varieties by infinite groups. The construction of $\Tw$ is not specific to the pentagram map and could have other applications. In fact, like Mumford's moduli space of closed polygons \cite{GIT}, the moduli space $\Tw$ admits a GIT semistable compactification with an explicit combinatorial description \cite{weinreich2021git}.

The next step is to compute formulas for the pentagram map. Then we derive a Lax representation with spectral parameter, which is an embedding of the dynamics into a matrix group. The formulas and Lax representation have already been derived in various guises in the literature, but our setup is a little different than the usual one, so we include full detail.

From this point, our proof follows the structure of Soloviev's proof of complex integrability of the pentagram map, modified to allow for positive characteristic \cite{MR3161305}. The characteristic polynomial of the Lax representation gives us a spectral curve. The next step is to show that the spectral curve is an integral curve (in the scheme-theoretic sense) of genus $n - 1$ when $n$ is odd and $n - 2$ when $n$ is even. These computations are technical in nature.

Then we construct the direct spectral transform, the birational map $\delta \colon \Tw \dashrightarrow \mathcal{A}$ of Theorem \ref{thm_main_1}. For the construction of $\delta$, the argument in \cite{MR3161305} goes through essentially without changes.

The brunt of the extra work in characteristic $p$ is the computation of the genus of the spectral curve, which is key to ensuring that the Jacobian has the right dimension. Since the genus of the spectral curve can change after reducing modulo $p$, we need to make sure that these curves have a kind of good reduction.

\subsection{Road map}

Section \ref{sect_related_work} describes related work. In Section \ref{sect_moduli}, we construct the moduli space $\Tw$ of twisted $n$-gons. In Section \ref{sect_formulas}, we derive formulas for the pentagram map. In Section \ref{sect_lax}, we derive the Lax representation. In Section \ref{sect_spectral_curve}, we analyze the spectral curve. In Section \ref{sect_spectral_transform}, we construct the direct spectral transform, finish the proof of Theorem \ref{thm_main_1}\new{, and study height growth}. In Section \ref{sect_dynamical_domains}, we formulate a conjecture that orbits of the pentagram map over a finite field almost always hit the degeneracy locus of the map.

\subsection{Acknowledgments}
The author thanks his advisor, Joe Silverman, for many hours of discussion of this project and for a careful reading of the manuscript. Further thanks to Dan Abramovich, Niklas Affolter, Ron Donagi, Sarah Griffith, Brendan Hassett, Boris Khesin, Anton Izosimov, John Roberts, Richard Schwartz, and Serge Tabachnikov for helpful conversations. The author was supported by an NSF Graduate Research Fellowship.

\new{An earlier version of this work appeared as a chapter of the author's PhD thesis \cite[Chapter 3]{weinreich_thesis}.}

\section{Related work} \label{sect_related_work}

There are many classical examples of continuous-time integrable systems originating in physics, but discrete-time examples are few and far between. Finding new examples is a major research area \cite{MR2088466}.

\new{The main antecedents of our result are the Liouville-Arnold integrability of the pentagram map \cite{MR2679816}, Soloviev's proof of complex integrability \cite{MR3161305} and Izosimov's study of the pentagram map via difference operators \cite{MR4430020}.}

\new{
The special case of our Theorem \ref{thm_main_1} where the base field $k$ is $\C$ recovers the main theorems of Soloviev \cite[Theorem A, Theorem B]{MR3161305}. However, the setup in the two papers is different, as we now explain. When $n$ is even, our Theorem \ref{thm_main_1} explains that the invariant subvarieties generically have two irreducible components, each isomorphic to a Jacobian, and the pentagram map sends each component into the other, isomorphically. The two components correspond to two ways of marking certain special points on the spectral curve, which come from making a choice of square root; see Sections \ref{sect_spectral_curve} and \ref{sect_spectral_transform}. In contrast, Theorem A of \cite{MR3161305} states that ``Each torus (Jacobian $J(\Gamma)$) is invariant for the pentagram map.'' Theorem B of \cite{MR3161305} describes the dynamics when $n$ is even as ``staircase-like'', that is, the pentagram map is not treated as a single-valued algebraic map, but rather depends on time, alternating between two translations on a single Jacobian. The reason for the discrepancy is that the two components of each invariant fiber have been identified in \cite{MR3161305}, by forgetting the marking. The choice of square root (hence the marking) flips upon application of the pentagram map, so the pentagram map is not a well-defined self-map of the single Jacobian appearing in \cite{MR3161305}.}

\new{We also fill a gap in the proof of complex integrability \cite[Theorem 2.9]{MR3161305}. This theorem concerns the singularities and genus of the generic spectral curve. To find the genus of the generic curve in a family, one needs an upper bound on the genus together with a ``one-point calculation" showing that the upper bound is achieved somewhere. This one-point calculation plays a role somewhat like checking the rank of the Poisson structure at a single point, as in \cite{MR2679816}. The argument in \cite[Theorem 2.9]{MR3161305} does not include the one-point calculation. Specifically, the proof asserts that the generic spectral curve, defined by a plane equation $R(k,z) = 0$, is nonsingular except at infinity. This is true, but difficult to justify; nonsingularity arguments usually depend on checking the nonvanishing of a resultant, but here the joint resultant of $R, \partial R / \partial k, \partial R / \partial z$ does in fact vanish, due to the singularity at infinity. We replace this assertion with several one-point calculations in Section \ref{sect_spectral_curve}; see in particular the casework depending on characteristic and the data in Tables \ref{table_resolution_of_singularities} and \ref{table_genus_computation_case_not_2n}.}

There are many ways to generalize the pentagram map; see, for instance, \cite{MR3534837, MR4430020, MR3118623}. The height growth in orbits of these generalized pentagram maps offers empirical evidence of integrability vs. non-integrability; see \cite{MR3282373}.

Theorem \ref{thm_main_1} describes the \emph{generic} behavior of the pentagram map. Many special classes of twisted $n$-gons have more idiosyncratic dynamics, including closed polygons \cite{MR3102478}, Poncelet polygons \cite{MR4460093}, and axis-aligned polygons \cite{MR3282366}. \new{Most recently, Schwartz has established a remarkable pentagram rigidity conjecture for the 3-diagonal map on centrally symmetric octagons \cite{schwartz2022pentagram}.}

There is also a substantial literature connecting the pentagram map to other fields, including cluster algebras \cite{MR3534837, MR2793031, MR3282370}, projective incidence theorems \cite{MR2721306}, Poisson-Lie groups \cite{MR3675462, MR4430020}, and integrable PDEs \cite{MR2679816}.

Our construction of the moduli space $\Tw$ of twisted $n$-gons follows an idea of Izosimov to take an appropriate quotient of a space of difference operators. This idea is introduced in \cite[Proposition 3.3]{MR4430020}, where the identification is shown to be a homeomorphism. We promote it to an algebraic isomorphism (Theorem \ref{thm_moduli}).

Singularity confinement, a feature of many discrete integrable systems, was explored in \cite{MR3104731}. Singularity confinement is closely related to the existence of a partial compactification on which the pentagram map becomes a morphism, which we construct in Theorem 1.1. 

\new{The corner invariants coordinatize the space of twisted $n$-gons by cross-ratios. Cross-ratios are a frequent source of compactifications in the theory of moduli spaces, for instance, the Naruki cross-ratio variety \cite{MR662660}.}

Little is known in general about the dynamics of rational maps on $\PP^n$ over finite fields; see the survey \cite[Section 18]{MR4007163}, and for arithmetic dynamics more generally, see \cite{ads}. Even in the simplest case, polynomials on $\PP^1$, we have only scattered pieces of the whole picture, and rational maps in higher dimension are even more complicated. For reversible maps and integrable systems, there are some probabilistic models for the statistics of the orbits \cite{MR2525820, siu_thesis}. Our work is motivated by the need for concrete examples of rational maps over finite fields for which the dynamics can be totally described.

The thesis \cite{kanki_thesis} collects some results on integrable systems over finite fields and, taking a more arithmetic dynamical angle, suggests viewing integrability over finite fields as a kind of $p$-adic singularity confinement, or ``almost good reduction.'' The other existing works on discrete integrable systems over finite fields focus on the construction of cellular automata with solitonic properties. This is also a nice perspective for our setting. Theorem \ref{thm_main_1} shows that the pentagram map over $\F_q$ defines an integrable cellular automaton on an alphabet of $q^2 + q + 2$ cell states, corresponding to the points of $\PP^2(\F_q)$ together with an extra state to represent degeneration of the map. The Toda molecule over $\F_{2^m}$ is studied in these terms in \cite{nakamura_mukaihira_finite_toda}. The articles \cite{MR2105640, MR2048090, MR1984011, MR2970774} study the discrete KdV and KP equations and the Hirota equation over finite fields as cellular automata. These articles restrict attention to genus 0 and 2 spectral curves, with a focus on special solutions which do not degenerate.

These articles apply the formulas of integrable systems in characteristic $0$ to finite fields. We emphasize that discrete integrability in characteristic $0$ does not imply the same over characteristic $p$. While the conserved quantities still exist, their algebraic independence is not guaranteed. The geometry of the spectral curve, its genus, and the application of the Riemann-Hurwitz formula are all characteristic-dependent. Indeed, by restricting attention to a subfamily of polygons where the spectral curve has worse singularities mod $p$, we can force the loss of algebro-geometric integrability in that family. This means the focus of proving integrability is on showing that the generic spectral curve has good reduction.

Ultradiscretization, or tropicalization, is a totally different idea for producing integrable systems valued in finite sets; see \cite{ultradiscretization}.

Some other surprising connections between integrable systems and number theory in finite characteristic are suggested in \cite{MR1837892}.

\section{The moduli space of twisted \texorpdfstring{$n$}{n}-gons} \label{sect_moduli}

\new{In this paper, we study the pentagram map on twisted polygons. The space of twisted polygons is larger than the space of closed polygons, but is still finite-dimensional.}

\begin{defn} \label{def_twisted_ngon}
Let $n \geq 4$ be a positive integer.
A \emph{twisted $n$-gon} is a $\Z$-indexed sequence $(v_i)$ in $\PP^2$ with the properties:
\begin{itemize}
    \item There exists a projective transformation $M \in \PGL_3$ such that, for all $i$,
$$ Mv_i = v_{i+n}.$$
    \item A nondegeneracy condition: in each 5-tuple of consecutive points 
    $$(v_i, v_{i+1}, v_{i+2}, v_{i+3}, v_{i+4}),$$
    no 3 points are collinear, except possibly $v_i, v_{i+2}, v_{i+4}$.
\end{itemize}
The transformation $M$ is called the \emph{monodromy} of the twisted $n$-gon. By the nondegeneracy condition, the monodromy $M$ is unique. A \emph{closed polygon} is a twisted polygon for which $M = 1$.
The set of twisted $n$-gons is denoted $\TwParamNice$.
\end{defn}

\new{
\begin{remark}
There is variation in the literature in the definition of twisted $n$-gon. We chose our definition in order to get a convenient moduli space. The most frequently used definition, from \cite{MR2679816}, only requires consecutive triples to be in general position. This definition is too permissive for our purposes because then the geometric quotient does not exist. Another common definition asks for all the points to be in general position. But this is too strict for our setting, since working over $\bar{\F}_p$, there are no such sequences.
\end{remark}
}
\new{
\begin{prop}
Let $n \geq 4$. The set $\TwParamNice$ of twisted $n$-gons may be identified with a Zariski open subset of $(\PP^2)^n \times \PGL_3$ via the map
$$\TwParamNice \hookrightarrow (\PP^2)^n \times \PGL_3,$$
$$(v_i)_{i \in \Z} \mapsto (v_1, \hdots, v_n, M),$$
where $M$ is the unique matrix that sends the 4-tuple $(v_1, v_2, v_3, v_4)$ to $(v_{n + 1}, v_{n+2}, v_{n+3}, v_{n+4})$.
\end{prop}
\begin{proof}
We need only show that this map is invertible on a generic subset of the variety $(\PP^2)^n \times \PGL_3$. This is straightforward: for a generic choice of $(v_1, \hdots, v_n, M)$, the points $v_1, \hdots, v_n, Mv_1, Mv_2, Mv_3, Mv_4$ are in general linear position, and thus so are any consecutive $5$ points in the sequence
$$(\hdots, v_1, \hdots, v_n, Mv_1, \hdots, Mv_n, M^2 v_1, \hdots, M^2 v_n, \hdots).$$
This sequence is a twisted $n$-gon that corresponds to $(v_1, \hdots, v_n, M)$, since it satisfies the nondegeneracy condition of Definition \ref{def_twisted_ngon}.
\end{proof}
}
The pentagram map is invariant under projective transformations. So, we study the induced map on the moduli space of projective equivalence classes of twisted $n$-gons. If we view $\TwParamNice$ as an open subvariety of $(\PP^2)^n \times \PGL_3$, the $\PGL_3$-action on $\TwParamNice$ is given by
\begin{equation}
\label{eqn:actionofPGLonUn}
A \cdot (v_1, \dots, v_n, M) = (Av_1, \dots, Av_n, AMA^{-1}).
\end{equation}

\begin{defn} \label{def_twisted_moduli}
The \emph{moduli space of twisted $n$-gons}, denoted $\Tw$, is the quotient variety $\TwParamNice / \PGL_3$ for the action described by~\eqref{eqn:actionofPGLonUn}.
\end{defn}

Definition \ref{def_twisted_moduli} asserts the existence of a quotient in the category of varieties, but in general, such quotients may not exist. The main theorem of this section, Theorem \ref{thm_moduli}, equips $\Tw$ with a variety structure. We explicitly describe the coordinate ring of the moduli space $\Tw$, and we check that the quotient map to $\Tw$ is geometric in the sense of geometric invariant theory.

We briefly recall the basic notions and give a more detailed review in Section \ref{subsect_git_background}.

Informally, given a variety $V$ and group $G$, a \emph{categorical quotient}, denoted $V \GITQ G$, is a variety $V'$ and a map $V \to V'$ which has the typical categorical properties of a quotient. Categorical quotients do not always exist, and even when they do, they may not reflect the geometry of the orbits well. A categorical quotient $V' = V \GITQ G$ is called a \emph{geometric quotient} if the points of $V'$ classify $G$-orbits in $V$. When a geometric quotient exists, the notions of orbit space and categorical quotient are essentially the same, so we can speak of a variety structure on $V / G$.

\new{We now recall the definition of Schwartz's corner invariants, which have played an essential role in the study of the pentagram map \cite{MR2434454}. These functions were known to define a full set of coordinates on $\Tw$ as a manifold over $\R$; we extend this to show that they generate the coordinate ring of $\Tw$ as a variety over $k$.}

\begin{defn}
We define the cross-ratio of four points $v_1, v_2, v_3, v_4$ in $\PP^1$ with $v_1 \neq v_3$ and $v_2 \neq v_4$ in a slightly non-standard way, as follows. Choose any affine coordinate such that, computed in that coordinate, we have $v_1, v_2, v_3, v_4 \not\in \{ 0, \infty \}$. Then the cross-ratio is defined by the formula
$$[v_1, v_2, v_3, v_4] = \frac{(v_1 - v_2)(v_3 - v_4)}{(v_1 - v_3)(v_2 - v_4)}.$$
One can check that the result is independent of the choice of affine coordinate. (Frequently one sees the reciprocal of this quantity defined as the cross-ratio.)
\end{defn}

Given a twisted $n$-gon $v=(v_i)$, its \emph{left and right corner invariants}, denoted $x_i, y_i$, are defined by
$$x_i = [v_{i - 2}, v_{i - 1}, \overline{v_i v_{i+1}} \cap \overline{v_{i - 2} v_{i - 1}}, \overline{v_{i+1} v_{i+2}} \cap \overline{v_{i-2} v_{i-1}}],$$
$$y_i = [\overline{v_{i+1}v_{i+2}} \cap \overline{v_{i-2} v_{i-1}}, \overline{ v_{i-1} v_i} \cap \overline{ v_{i+1} v_{i+2}},v_{i+1}, v_{i+2}].$$
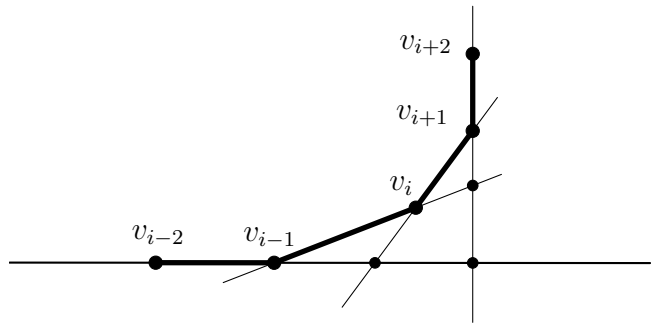
\begin{figure}[h]
\begin{center}
\begin{tikzpicture}[line cap=round,line join=round,x=2cm,y=2cm]
\clip(-8.078937962480818,0.5197631345776252) rectangle (-3.8210063920154718,3.5020930758952074);
\draw [line width=2pt] (-7.107359808810209,1.6306020673469557)-- (-6.32,1.63);
\draw [line width=2pt] (-6.32,1.63)-- (-5.378359719798769,1.9964041665062287);
\draw [line width=2pt] (-5.378359719798769,1.9964041665062287)-- (-5,2.5079);
\draw [line width=2pt] (-5,2.5079)-- (-5,3.01963);
\draw [line width=0.8pt,domain=-8.078937962480818:-3.8210063920154718] plot(\x,{(--1.2795914227278793-0.0006020673469557813*\x)/0.7873598088102085});
\draw [line width=0.4pt] (-5.868094817254599,1.3343425125924593)-- (-4.835794176183923,2.7298860897507966);
\draw [line width=0.4pt] (-5,3.3365597589392766)-- (-5,1.2335010635981716);
\draw [line width=0.4pt] (-6.658647430388726,1.4982279931307303)-- (-4.80885185544657,2.218006893942777);
\draw [fill=black] (-7.107359808810209,1.6306020673469557) circle (2.5pt);
\draw[color=black] (-7.090013354517834,1.8109208811540152) node {$v_{i-2}$};
\draw [fill=black] (-6.32,1.63) circle (2.5pt);
\draw[color=black] (-6.347764322923144,1.7975870662151885) node {$v_{i-1}$};
\draw [fill=black] (-5.378359719798769,1.9964041665062287) circle (2.5pt);
\draw[color=black] (-5.47439944442999,2.1442662546246836) node {$v_{i}$};
\draw [fill=black] (-5,2.5079) circle (2.5pt);
\draw[color=black] (-5.334387163489908,2.606505172504011) node {$v_{i+1}$};
\draw [fill=black] (-5,3.01963) circle (2.5pt);
\draw[color=black] (-5.312164138591863,3.0776333003425558) node {$v_{i+2}$};
\draw [fill=black] (-5.649772462856632,1.6294874997293114) circle (2pt);
\draw [fill=black] (-5,1.6289906407603119) circle (2pt);
\draw[fill=black] (-5,2.144) circle (2pt);
\end{tikzpicture}
\caption{\new{The left corner invariant $x_{i}$ is defined as the cross-ratio of the four points on the pictured horizontal line, and the right corner invariant $y_i$ is defined as the cross-ratio of the four points on the pictured vertical line. We can informally think of the corner invariants as providing coordinates for $v_i$ in the axes determined by $v_{i-2}$, $v_{i-1}$, $v_{i+1}$, and $v_{i+2}$.}}
\end{center}
\end{figure}

These quantities are $\PGL_3$-invariant functions of $v$. We consider the morphism
$$\TwParamNice \to (\PP^1 \smallsetminus \{0, 1, \infty\})^{2n},$$
$$v \mapsto (x_1(v), \dots, x_n(v), y_1(v), \dots, y_n(v)).$$
We now present the main theorem of this section, which says that this morphism defines a geometric quotient (i.e. a variety with nice quotient properties).

\begin{thm} \label{thm_moduli}
The geometric quotient $\Tw = \TwParamNice / \PGL_3$ exists, and there is an isomorphism of varieties
$$\Tw \xrightarrow{\sim} (\PP^1 \smallsetminus \{0, 1, \infty \})^{2n},$$
$$v \mapsto (x_1(v), \dots, x_n(v), y_1(v), \dots, y_n(v)).$$
\end{thm}

We prove a more detailed version of this theorem as Proposition \ref{prop_geoquoladder}, later in this section. 
Theorem \ref{thm_moduli} implies that any algebraic, projectively invariant condition on twisted $n$-gons can be expressed in terms of the corner invariants. For instance, one can show that the points $v_{i - 2}, v_i, v_{i+2}$ are collinear if and only if $x_i y_i = 1$.

The real part of $\Tw$ was studied in \cite{MR2679816}, where it was shown that the corner invariants provide a diffeomorphism to $\R^{2n}$. To do something similar algebraically, we have to take the quotient of a variety by an infinite group, using GIT.

The proof of Theorem \ref{thm_moduli} is complicated. The motivating idea is as follows. In any sequence of five points in the plane, we can describe coordinates for the third point in terms of the first, second, fourth, and fifth points. Take consecutive subsequences of length 5 in the bi-infinite sequence defining a twisted $n$-gon. Because of the monodromy, the resulting coordinates are $n$-periodic, thus giving $2n$ functions that coordinatize the space of twisted $n$-gons.

\begin{remark}
The existence of the geometric quotient $\Tw$ follows immediately from the main results of \cite{weinreich2021git}. However, that proof does not produce the explicit coordinatization by corner invariants, and we use that coordinatization to study the pentagram map.
\end{remark}

\begin{remark}
Another coordinate system, the so-called \emph{$ab$-coordinates}, has also been used widely in the study of the pentagram map, for instance in \cite{MR2679816}. These work well over $\R$, but their definition uses the isomorphism $\SL_3(\R) \cong \PGL_3(\R)$, and they are not well-defined over an arbitrary base field. Further, our proof of Theorem \ref{thm_moduli} explains the algebraic origin of the corner invariants.
\end{remark}

\subsection{Background on geometric invariant theory} \label{subsect_git_background}

We cite some standard theorems in geometric invariant theory. For a development of these ideas, see \cite[Chapter 6]{MR2004511}.

\begin{defn}
Suppose that a group $G$ acts algebraically on a variety $V$, with action $\alpha: G \times V \to V$. A \emph{categorical quotient} is a variety $V'$ and a $G$-invariant morphism $\chi: V \to V'$, such that for every variety $V''$, every $G$-invariant morphism $V \to V''$ factors through $\chi$ uniquely. We denote a categorical quotient $V'$ by $V \GITQ G$. The quotient depends on the action, but this notation suppresses it. If we want to be specific about the action, we write $V \GITQ \alpha$. We will also use the notation $V \GITQ G$ and $V \GITQ \alpha$ to refer to the variety without the attached data of $\chi$.

By a universal property argument, if a categorical quotient exists, it is unique up to unique isomorphism. When we write $V \GITQ G$, the argument will always provide a specific construction of the categorical quotient.

A \emph{geometric quotient}, denoted $V / G$, is a categorical quotient $V \GITQ G$ with the following property: if $v_1, v_2 \in V$ have the same image in $V \GITQ G$, then there exists some $g \in G$ such that $g \cdot v_1 = v_2$.
\end{defn}

Categorical quotients do not always exist, much less geometric quotients. But for a large class of groups, the \emph{geometrically reductive groups}, categorical quotients of affine varieties do exist. Any algebraic subgroup of a general linear group $\GL_d$ is geometrically reductive, regardless of base field.

\begin{thm}[Nagata, Mumford]
\label{thm_nagata}
Let $G$ be a geometrically reductive group acting on an affine $k$-variety $V$. Let $\shO(V)$ be the $k$-algebra of global sections. Then the subalgebra $\shO(V)^G$ of $G$-invariant functions is finitely generated over $k$. Let $V' = \Spec \shO(V)^G$. The canonical morphism $\psi: V \to V'$ is a categorical quotient. Thus
$$ V \GITQ G \cong \Spec \shO(V)^G.$$
\end{thm}
The first claim, about finite generation, is due to Nagata \cite{MR179268}. The scond claim is due to Mumford \cite{MR1304906}. For a proof, see \cite[Theorem 6.1]{MR2004511}.

We use $\Gm$ to denote the multiplicative group scheme over $k$.

\begin{example}
Consider $\Gm$ acting on $\A^1$ by $t \cdot v = tv$. The categorical quotient $\A^1 \GITQ \Gm$ exists by Theorem \ref{thm_nagata}, and $\shO(\A^1)^{\Gm} = k$, so the categorical quotient is a point. Since $0$ and $1$ are in different orbits, but have the same image, the quotient is not geometric. 
\end{example}

\subsection{The dual space} \label{sub_thedualspace}

In this section, we prove Theorem \ref{thm_moduli}. We describe a map which takes a twisted $n$-gon to a kind of dual data: the list of relations satisfied by consecutive sets of four points. The objects in the proof will be used again to derive formulas for the pentagram map in Section \ref{sect_formulas}.

We introduce two spaces $\LeftSpace$ and $\RightSpace$, along with actions $\alpha$ and $\beta$, which are closely related to the action of $\PGL_3$ on $\TwParamNice$. The existence of the geometric quotient $\RightSpace / \beta$ is easier to establish, and we use its structure to show existence of the geometric quotient $\LeftSpace / \alpha$, which in turn gives us existence of $\TwParamNice / \PGL_3$ and the explicit description. The relationships between the various objects are depicted in the following diagram.

\begin{center}
\begin{tikzcd}
& (\PP^2)^n \times \PGL_3 & (\A^{3})^n \times \GL_3 & (\A^4)^n \\
& \TwParamNice \arrow[u, hook] \arrow[d, twoheadrightarrow] & \LeftSpace \arrow[u, hook] \arrow[d, twoheadrightarrow, "\chi"] \arrow[r, "\Dualize", shift left] \arrow[l, twoheadrightarrow, "\pi"] & \RightSpace \arrow[u, hook] \arrow[d, twoheadrightarrow, "\chi'"] \arrow[l, "\DualizeBack", shift left] \\
\Tw \arrow[r, equals] & \TwParamNice / \PGL_3 \arrow[r, leftrightarrow, "\sim"] & \LeftSpace / \acLeft \arrow[r, hookrightarrow] & \RightSpace / \beta \arrow[r, leftrightarrow, "\sim"] & \Gm^{2n}
\end{tikzcd}
\end{center}

We now define
$$\LeftSpace \subset (\A^3 \smallsetminus \{0\})^n \times \GL_3.$$
For all $1 \leq i \leq n$, let $v_i$ be the projection to the $i$-th copy of $\A^3$. Let $M$ be the $\GL_3$-coordinate. For all $i > n$, define $v_{i} = Mv_{i-n}$. Let $\LeftSpace$ be the open subset defined by the condition that, for all $i$ with $1 \leq i \leq n$, the five consecutive vectors $v_i, v_{i+1}, v_{i+2}, v_{i+3}, v_{i + 4}$ are nonzero, and no 3 of them are coplanar, except possibly $v_i, v_{i + 2}, v_{i + 4}$. Observe that $\LeftSpace$ consists of the elements of $(\A^3 \smallsetminus \{0\})^n \times \GL_3$ which are sent by projectivization to $\TwParamNice$. Let $\pi \colon \LeftSpace \to \TwParamNice$ denote projectivization.

We embed $\LeftSpace$ in the space of $3 \times (n + 3)$ matrices via the map
\[
\begin{bmatrix}
\mid & \mid & \mid & & \mid & \mid & \mid & \mid & \mid & \mid \\
v_1 & v_2 & v_3 & \hdots & v_{n-2} & v_{n-1} & v_n & Mv_1 & Mv_2 & Mv_3 \\
\mid & \mid & \mid & & \mid & \mid & \mid & \mid & \mid & \mid
\end{bmatrix}.
\]
Each set of four vectors in $\A^3$ satisfies some nontrivial relation. If no 3 of the vectors are coplanar, then each coefficient in the relation is nonzero. Thus, given an element $\theta \in \LeftSpace$, there exist values $a_i, b_i, c_i, d_i \in k^*$, such that $\theta$ is annihilated on the right by the $(3 + n) \times n$ matrix
\begingroup
\small
\[
\begin{bmatrix}
a_1 & \\
b_1 & a_2 \\
c_1 & b_2 \\
d_1 & c_2 \\
    & d_2 \\
& & & \ddots \\
& & & & a_{n-1} & \\
& & & & b_{n-1} & a_n \\
& & & & c_{n-1} & b_n \\
& & & & d_{n-1} & c_n \\
& & & &         & d_n
\end{bmatrix}.
\]
\endgroup

Let 
$$\RightSpace = (\Gm^4)^n.$$
The variety $\RightSpace$ parametrizes matrices of the above form. Let the coordinates on the $i$-th copy of $\Gm^4$ in $\RightSpace$ be $a_i, b_i, c_i, d_i$. We extend these definitions to be $n$-periodic, by the rule that for all $i > n,$
$$a_{i} = a_{i - n}, \quad b_{i} = b_{i - n}, \quad c_{i} = c_{i - n}, \quad d_{i} = d_{i - n}.$$
We define $\Dualize \colon \LeftSpace \to \RightSpace$ by
\[
\begin{aligned}
a_i &= \det \begin{bmatrix} \mid & \mid & \mid \\ v_{i+1} & v_{i+1} & v_{i+3} \\ \mid & \mid & \mid \\ \end{bmatrix},
\qquad
&b_i &= -\det\begin{bmatrix} \mid & \mid & \mid \\ v_{i} & v_{i+2} & v_{i+3} \\ \mid & \mid & \mid \\ \end{bmatrix}, \\
c_i &= \det \begin{bmatrix} \mid & \mid & \mid \\ v_{i} & v_{i+1} & v_{i+3} \\ \mid & \mid & \mid \\ \end{bmatrix},
\qquad
&d_i &= -\det\begin{bmatrix} \mid & \mid & \mid \\ v_{i} & v_{i+1} & v_{i+2} \\ \mid & \mid & \mid \\ \end{bmatrix}. \\
\end{aligned}
\]
Set
$$G = \Gm^{n+1} \times \GL_3.$$
We set notation for the coordinate on each factor.
\begin{itemize}
    \item For each $i$ where $1 \leq i \leq n$, let $\eta_i$ be a coordinate on the $i$-th copy of $\Gm$.
    \item Let $\xi$ be a coordinate on the $(n+1)$-th copy of $\Gm$.
    \item Let $A$ be the coordinate on the $\GL_3$ factor.
\end{itemize}
We now define an action
\begin{equation} \label{eq_left_action_def}
\alpha \colon G \times \LeftSpace \to \LeftSpace.
\end{equation}
\begin{itemize}
    \item For each $i$ in $1 \leq i \leq n$, the $\eta_i$ coordinate scales $v_i$.
    \item The $\xi$ coordinate scales the $\GL_3$ factor of $\LeftSpace$.
    \item The $A$ coordinate acts by
\end{itemize}
$$A \cdot (v_1, \dots, v_n, M) = (Av_1, \dots , Av_n, AMA^{-1}).$$

After writing an element of $\LeftSpace$ in matrix form, the scalings correspond to coordinatewise multiplication by various matrices:
\begin{itemize}
    \item For $i = 1, 2, 3$, the $\eta_i$ coordinate simultaneously scales columns $i$ and $n+i$. For each $i$ where $4 \leq i \leq n$, the $\eta_i$ coordinate only scales column $i$.
    \item The $\xi$ coordinate scales columns $n+1$, $n+2$, and $n+3$.
    \item The $A$ coordinate acts by change of basis.
\end{itemize}

We now define an action
$$\acRight: \Gm^{2n+1} \times \RightSpace \to \RightSpace.$$
We use the following coordinates on $\Gm^{2n+1}$.
\begin{itemize}
    \item For each $i$ where $1 \leq i \leq n$, let $\kappa_i$ be a coordinate on the $i$-th copy of $\Gm$.
    \item For each $i$ where $1 \leq i \leq n$, let $\rho_{i}$ be a coordinate on the $(n+i)$-th copy of $\Gm$.
    \item Let $\epsilon$ be a coordinate on the $(2n + 1)$-th copy of $\Gm$.
\end{itemize}
Viewing $\RightSpace$ as a space of matrices, the action of each factor is by a coordinatewise multiplication.
\begin{itemize}
    \item For each $i$ in $1 \leq i \leq n$, the $\kappa_i$ coordinate acts by scaling column $i$.
    \item For each $i$ in $1 \leq i \leq 3$, the $\rho_i$ coordinate acts by simultaneously scaling rows $i$ and $i + n$. For $4 \leq i \leq n$, the $\rho_i$ coordinate acts by scaling row $i$.
    \item The $\epsilon$ coordinate acts by a coordinatewise multiplication by
    \[
\begin{bmatrix}
1 & \\
1 & 1 \\
1 & 1 \\
1 & 1 \\
  & 1 \\
  &   & \ddots \\
& & & 1 & & & \\
& & & 1 & \epsilon &  & \\
& & & 1 & \epsilon & \epsilon & \\
& & & 1 & \epsilon & \epsilon & \epsilon \\
& & &        & 1 & 1 & 1 \\
& & &        &         & 1 & 1 \\
& & &        &         &         & 1
\end{bmatrix}.
\]
\end{itemize}

\new{Izosimov found formulas for the corner invariants of a twisted polygon in terms of $a_i, b_i, c_i, d_i$ \cite[Proposition 2.10]{MR4460093}. Without referring to corner invariants, we prove that those expressions, denoted $x'_i, y'_i$, generate the ring of invariants of $\beta$. We show later, in Lemma \ref{lemma_corner_invariants}, that these expressions agree with the corner invariants.}

\begin{prop} \label{prop_geoquo}
The geometric quotient $\RightSpace / \acRight$ exists. It is given explicitly by
$$\RightSpace / \acRight = \Spec k[{(x'_i)}^{\pm 1}, {(y'_i)}^{\pm 1}] \cong \Gm^{2n}, $$
with the natural projection
$$\chi' \colon \RightSpace \to \RightSpace / \acRight$$
where for each $i = 1, \hdots, n,$
    $$ x'_i \colonequals a_{i-1} c_{i - 2} {b_{i-1}}^{-1} {b_{i-2}}^{-1},$$
    $$ y'_i \colonequals d_{i-2} b_{i-1} {c_{i-2}}^{-1} {c_{i-1}}^{-1}.$$
\end{prop}

\begin{proof}
There are three parts to the proof: constructing a categorical quotient $\RightSpace \GITQ \acRight$, calculating its coordinate ring, and showing that the quotient is geometric.
\par(1)\enspace 
To see that $\RightSpace \GITQ \acRight$ exists, observe that $\RightSpace$ is an affine variety and that $(\Gm)^{2n + 1} \times \GL_3$ is reductive. By Nagata's Theorem (Theorem \ref{thm_nagata}), the spectrum of the ring of invariants $k[\RightSpace]^\acRight$ is a categorical quotient.
\par(2)\enspace 
Let $x_i, y_i$ be defined in the statement of the theorem. We claim that
    $$ k[\RightSpace]^\acRight = k[{(x'_i)}^{\pm 1}, {(y'_i)}^{\pm 1}].$$
    The $x'_i, y'_i$ are invariants by inspection, so we just have to show that they generate the ring $k[\RightSpace]^\acRight$.
    
    Each of the $2n + 1$ actions of $\Gm$ puts a grading on $k[\RightSpace]$. Then $k[\RightSpace]^\acRight$ is the intersection of the 0-graded part for each grading. Further, the invariant ring is generated by monomials in the $a_i, b_i, c_i, d_i$, because the ring of invariants for each grading considered separately is generated by monomials. Now, we argue that every invariant monomial is of the form $\prod_i {(x'_i)}^{e_i} {(y'_i)}^{f_i}$, where $e_i, f_i \in \Z$. If $\mu$ is an invariant monomial, then we can divide by an appropriate power of the $x_i$ and $y_i$ to get an invariant monomial that has $a_i$-degree 0 and $d_i$-degree 0. The resulting invariant is a monomial in just the $b_i, c_i$.  Since each $\kappa_i$-grading is 0, we have $\deg_{c_i} \mu = -\deg_{b_i} \mu$. Since each $\rho_i$-grading is 0, we have $\deg_{b_i} \mu = - \deg_{c_{i-1}} \mu$ for $i > 1$, and $\deg_{b_1} \mu = -\deg_{c_n} \mu$. 
    Write 
    $$\mu_0 = \frac{b_1 b_2 \hdots b_n}{c_1 c_2 \hdots c_n}.$$
    We have shown that $\mu$ is a power of $\mu_0$. And the $\epsilon$-grading gives $\mu_0$ degree 1, so $\mu = 1$.
    
    Further, we claim the spectrum is $\Gm^{2n}$. This is true if the $x'_i, y'_j$ are algebraically independent. This is clear, since a distinct $a_i$ or $d_i$ appears in the definition of each, and these have no relations. 
\par(3)\enspace 
To prove that the quotient is geometric, we must check that invariant functions distinguish between orbits. Suppose $w, w' \in \RightSpace$ satisfy $\chi'(w) = \chi'(w')$. First, by an $\epsilon$-scaling, we can replace $w$, $w'$ by elements such that $\mu_0(w) = \mu_0(w')= 1$. Then, by scaling with $\kappa_1, \rho_1, \hdots, \kappa_2, \rho_2, \hdots, \kappa_{n}$, we can replace $w$ by an element such that
    $$b_1(w) = 1, \quad c_1(w) = 1, \quad b_2(w) = 1, \quad c_2(w) = 1, \quad \hdots, \quad b_n(w) = 1.$$
    We do the same for $w'$. Since the scalings by $\kappa_i$ and $\rho_i$ hold $\mu_0$ invariant, we deduce that $c_n(w) = c_n(w') = 1$ as well. An element of $\RightSpace$ with all $b_i, c_i = 1$ is determined by the values of $x'_i, y'_i$, so $w = w'$.
\end{proof}

\begin{prop} \label{prop_dualizing_orbits_I}
If the points $\theta_1, \theta_2 \in \LeftSpace$ are in the same $\acLeft$-orbit, then $\Dualize(\theta_1), \Dualize(\theta_2)$ are in the same $\acRight$-orbit.
\end{prop}

\begin{proof}
Recall that we defined action $\alpha \colon G \times \LeftSpace \to \LeftSpace$ just after \eqref{eq_left_action_def} as the product of actions by scaling by $\eta_i$, $1 \leq i \leq n$, scaling by $\zeta$, and conjugation by $A$. Thus we can reduce the claim to the cases that $\theta_2$ is obtained by applying just one of the $\eta_i$, $\zeta$, or $A$ actions.
\begin{itemize}
    \item Suppose that $\theta_2$ is obtained from $\theta_1$ by $\eta_i$. Then $\Dualize(\theta_2)$ is obtained from $\Dualize(\theta_1)$ by $\kappa_{i-3}, \kappa_{i-2}, \kappa_{i-1}, \kappa_i$, indices taken mod $n$, then applying the $\rho_i$-scaling by $(\eta_i)^{-1}$.
    \item Suppose that $\theta_2$ is obtained from $\theta_1$ by applying $A$. Then $\Dualize(\theta_2)$ is obtained from $\Dualize(\theta_1)$ by scaling every entry by $\det(A)$. This can be accomplished by scaling each column individually with the $\kappa_i$.
    \item Suppose that $\theta_2$ is obtained from $\theta_1$ by applying $\xi$. Then $\Dualize(\theta_2)$ is obtained from $\Dualize(\theta_1)$ by applying a coordinatewise multiplication by
\begingroup
\small
\[
\begin{bmatrix}
1 & \\
1 & 1 \\
1 & 1 \\
1 & 1 \\
  & 1 \\
  &   & \ddots \\
& & & 1 & & & \\
& & & 1 & \xi &  & \\
& & & 1 & \xi & \xi^2 & \\
& & & 1 & \xi & \xi^2 & \xi^3 \\
& & &        & 1 & \xi & \xi^2 \\
& & &        &         & \xi & \xi^2 \\
& & &        &         &         & \xi^2
\end{bmatrix}.
\]
\endgroup
This can be achieved using $\kappa_{n-1}, \kappa_{n}$, and $\epsilon$.
\end{itemize}
\end{proof}

\new{So far, we have described the invariants $x'_i, y'_i$ of the action $\beta$ in purely algebraic terms. In fact, they agree with the corner invariants.}

\begin{lemma}[{\new{\cite[Lemma 2.10]{MR4460093}}}] \label{lemma_corner_invariants}
The map $\chi' \circ \Dualize$ takes an element $v \in \LeftSpace$ to the corner invariants of $\pi(v)$.
\end{lemma}

\new{\begin{proof}
The idea is to compute $\chi' \circ \Delta$ for a carefully chosen element of $\LeftSpace$ that is $\alpha$-equivalent to $v$, then observe that $\alpha$-equivalence does not change the corner invariants or $\chi' \circ \Delta$, by Proposition \ref{prop_geoquo} and Proposition \ref{prop_dualizing_orbits_I}. See {\cite[Lemma 2.10]{MR4460093}} or \cite[Lemma 3.2.10]{weinreich_thesis} for two different approaches to the computation.
\end{proof}
}

\begin{prop}
\label{prop_dualizing_orbits_II}
Let $\theta, \theta' \in \LeftSpace$. If $\Dualize(\theta), \Dualize(\theta')$ are in the same $\acRight$-orbit, then $\theta, \theta'$ are in the same $\acLeft$-orbit.
\end{prop}

\begin{proof}
We may use moves in $\alpha$ to assume that $\theta$ and $\theta'$ agree in $v_1, v_2, v_4, v_5$. Let $u, u'$ be the image of $\theta, \theta'$ in $\TwParamNice$, respectively. By Propositions \ref{prop_geoquo} and \ref{lemma_corner_invariants}, the corner invariants of $u$ and $u'$ agree. We claim that $u_3 = u'_3$. Indeed, the value of the corner invariant $x_3(u)$ determines the point of intersection of $\overline{u_3 u_4}$ with $\overline{u_1 u_2}$, which gives us a line on which $u_3$ must lie. The value of $y_3(u)$ similarly picks out a second line on which $u_3$ must lie. This determines $u_3$. Thus we know the relative positions of each consecutive group of 5 points in $u$, which determines all of $u$.
\end{proof}

Finally, we prove the following more detailed version of Theorem \ref{thm_moduli}.

\begin{prop} \label{prop_geoquoladder}
The geometric quotients $\TwParamNice / \PGL_3$ and $\LeftSpace / \acLeft$ exist, and are isomorphic. The dualization map $\Delta$ descends to an embedding
$$\LeftSpace / \acLeft \hookrightarrow \RightSpace / \acRight,$$
and there is an isomorphism
$$\LeftSpace / \acLeft \cong (\PP^1 \smallsetminus \{0,1, \infty\})^{2n}.$$
\end{prop}
\begin{proof}

Let $\RightSpace^\circ \subset \RightSpace$ be the subset where all $x'_i, y'_i \not\in \{0, 1\}$. 
The geometric quotient 
$$\RightSpace / \acRight \cong \Gm^{2n}$$
restricts to a geometric quotient
$$\chi' \colon \RightSpace^\circ \to (\PP^1 \smallsetminus \{0, 1, \infty\})^{2n}.$$
Consider the map 
$$\Delta' \colon \RightSpace^\circ \to \LeftSpace,$$
where $\Delta'(w)$ has $v_1, v_2, v_3$ at the standard basis vectors, and each remaining vector $v_i$ for $4 \leq i \leq n$ is determined recursively by the relation
$$a_{i - 3} v_{i - 3} + b_{i  - 3} v_{i - 2} + c_{i - 3} v_{i - 1} + d_{i - 3} v_i = 0.$$
We formally compute $v_{n + 1}, v_{n + 2}, v_{n + 3}$ in the same way, and then
define $M$ to be the matrix $[v_{n+1} v_{n+2} v_{n+3}]$.

Define 
$$\chi \colon \LeftSpace \to (\PP^1 \smallsetminus \{0, 1, \infty\})^{2n},$$
$$v \mapsto (x_1(v), \hdots, x_n(v), y_1(v), \hdots, y_n(v)).$$
We claim that $\chi$ satisfies the necessary universal property to be a categorical quotient $\LeftSpace / \acLeft$. The argument is the following elementary diagram chase.

We know that
$$\chi \circ \Delta = \chi',$$
$$\chi' \circ \Delta' = \chi.$$
Suppose that we are given a variety $V'$ and an $\acLeft$-invariant map $\nu \colon \LeftSpace \to V'$. We want to show that there exists a unique map 
$$\nu' \colon (\PP^1 \smallsetminus \{0, 1, \infty\})^{2n} \to V'$$
such that
$$\nu' \circ \chi = \nu.$$
Evidently
$$\nu \circ \Delta'$$
is $\acRight$-invariant, so from the universal property of the categorical quotient $\RightSpace^\circ / \acRight$, there exists a unique map $\nu'$ such that
$$\nu' \circ \chi' = \nu \circ \Delta'.$$
Then composing each side of the equation with $\Delta$ on the right gives the desired equation
$$\nu' \circ \chi = \nu.$$
Further, if some other $\nu''$ satisfies
$$\nu'' \circ \chi = \nu,$$
then composing with $\Delta'$ on the right gives
$$\nu'' \circ \chi' = \nu \circ \Delta',$$
so again by the universal property, 
$$\nu'' = \nu'.$$

Suppose that $v, v' \in \LeftSpace$, and suppose $\chi(v) = \chi(v')$.
Then $\Dualize(v)$ and $\Dualize(v')$ are in the same $\acRight$-orbit. So, by Proposition \ref{prop_dualizing_orbits_II}, we know that $v$ and $v'$ are in the same $\acLeft$-orbit. 

Now that we know that the geometric quotient $\LeftSpace / \alpha$ exists, we compute it another way. Since the $\Gm$ factors act by scaling the $\A^3$'s and the $\GL_3$, we can write $\LeftSpace / \alpha$ as
$$\TwParamNice / \GL_3.$$
Since scaling is trivial on $\TwParamNice$,
$$\TwParamNice / \GL_3 \cong \TwParamNice / \PGL_3.$$
This concludes the proof of Proposition~\ref{prop_geoquoladder}.
\end{proof}

\begin{proof}[Proof of Theorem \ref{thm_moduli}]
Immediate from Proposition \ref{prop_geoquoladder}.
\end{proof}

\section{Formulas for the Pentagram Map} \label{sect_formulas}

\begin{defn}
\new{
Fix $n \geq 4$. The \emph{pentagram map on parameter space} is the rational map
$$ F \colon \TwParamNice \dashrightarrow \TwParamNice $$
that sends a twisted $n$-gon $(v_i)_{i \in \Z}$ to the twisted $n$-gon $(w_i)_{i \in \Z}$, where $w_i$ is the intersection of the diagonals $\overline{v_{i-1} v_{i+1}}$ and $\overline{v_i v_{i+1}}$. See Figure \ref{fig_twisted_pentagram}. The map $F$ respects the action of $\PGL_3$ on $\TwParamNice$ defined by \eqref{eqn:actionofPGLonUn}.}

\new{
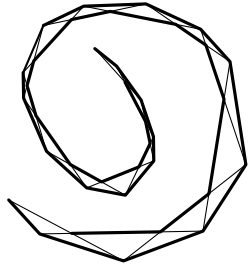
\begin{figure}[h]
    \centering
\begin{tikzpicture}[line cap=round,line join=round,x=0.2cm,y=0.2cm]
\clip(-7.009647700753595,-10.564578724198437) rectangle (19.718881204528017,9.557386719897401);
\draw [very thick] (-2.891866356634992,-3.9970847824212115)-- (-0.8303782756215154,-6.2560024930770926)-- (4.739541359208339,-8.050161461492808)-- (10.020898497495287,-6.069265836311675)-- (12.87299080185351,-1.6457336195860943)-- (11.816789921380966,5.205754090593068)-- (9.86521153117463,7.359503066464103)-- (6.230561212128679,9.044828974864867)-- (2.091444811951597,8.893518225906615)-- (-0.5878871914892096,7.609064497931349);
\draw [very thick] (-1.9303890476208236,4.47699994338703)-- (-1.9011949669064991,2.7756549521154943)-- (-0.38020773883920667,-0.7553227755709966)-- (2.2942281882985545,-3.192606172913549)-- (4.860877815955236,-3.6792572676396462)-- (6.751509865395017,-1.3991382203323373)-- (6.728981535734588,0.21527465043160743)-- (5.908659220329106,2.6466864407898383)-- (4.2815548756383,4.903414786609194)-- (2.8312234223489767,6.070758613727138);
\draw [line width=0.8pt] (-0.5878871914892096,7.609064497931349)-- (-1.9303890476208236,4.47699994338703);
\draw [thin] (-2.891866356634992,-3.9970847824212115)-- (4.739541359208339,-8.050161461492808)-- (12.87299080185351,-1.6457336195860943)-- (9.86521153117463,7.359503066464103)-- (2.091444811951597,8.893518225906615)-- (-1.9303890476208236,4.47699994338703)-- (-0.38020773883920667,-0.7553227755709966)-- (4.860877815955236,-3.6792572676396462)-- (6.728981535734588,0.21527465043160743)-- (4.2815548756383,4.903414786609194);
\draw [thin] (-0.8303782756215154,-6.2560024930770926)-- (10.020898497495287,-6.069265836311675)-- (11.816789921380966,5.205754090593068)-- (6.230561212128679,9.044828974864867)-- (-0.5878871914892096,7.609064497931349)-- (-1.9011949669064991,2.7756549521154943)-- (2.2942281882985545,-3.192606172913549)-- (6.751509865395017,-1.3991382203323373)-- (5.908659220329106,2.6466864407898383)-- (2.8312234223489767,6.070758613727138);
\draw [very thick] (1.2925892642847916,-6.219468924805591)-- (7.193431812987472,-6.117922937850975)-- (10.417514202862458,-3.579221103171445)-- (11.418908453388788,2.7077626082062665)-- (10.217460587155392,6.304875769684928)-- (8.206635010051723,7.686793763290191)-- (3.857221413885827,8.545073492048338)-- (1.2799376057875334,8.002373405518382)-- (-1.2299046385924992,5.2462266959367785)-- (-1.6740279667298055,3.6117050903243446)-- (-1.0803931972430878,1.6080112813788412)-- (1.2010463036126293,-1.637484149795929)-- (3.278601682468304,-2.796525681919346)-- (5.763994833388017,-1.796482765932262)-- (6.510202924490622,-0.2408243741886965)-- (6.206809396765394,1.215515516729459)-- (4.837476977164966,3.8385245656892244);
\end{tikzpicture}
    \caption{The pentagram map $F$ on parameter space. The exterior bold line is part of the input twisted polygon, and the interior bold line is part of its image by the pentagram map.}\label{fig_twisted_pentagram}
\end{figure}
}

\new{The \emph{pentagram map on moduli space}, or just \emph{pentagram map}, is the rational map
$$ f \colon \Tw \dashrightarrow \Tw $$
induced on the moduli space of twisted $n$-gons up to the $\PGL_3$-action \eqref{eqn:actionofPGLonUn}. In other words, the map $f$ is the unique dominant rational map such that the following diagram commutes:}

\begin{center}
\begin{tikzcd}
\TwParamNice \arrow[r, dashed, "F"] \arrow[d, dashed] & \TwParamNice \arrow[d, dashed] \\
\Tw \arrow[r, dashed, "f"] & \Tw
\end{tikzcd}
\end{center}
\end{defn}

We showed in Theorem \ref{thm_moduli} that
$$ \Tw \cong (\PP^1 \smallsetminus \{0,1,\infty\})^{2n}.$$
\new{The coordinates on $\Tw$ are $x_i, y_i$ for each $i = 1, \hdots, n$. In these coordinates, the pentagram map has a straightforward formula, which we apply in Section \ref{sect_lax}. Define
$$f^*(x_i) = x_i \circ f, \quad f^*(y_i) = y_i \circ f.$$
This is just the usual algebro-geometric definition of pullback of a function by a rational map.}
\new{
\begin{prop}[\cite{MR2434454}] \label{prop_formulas}
The pentagram map $f \colon \Tw \dashrightarrow \Tw$ can be written in coordinates as
\begin{align*}
    f^*(x_i) &= x_{i + 1} \frac{1 - x_i y_i}{1 - x_{i+2}y_{i+2}},\\
    f^*(y_i) &= y_{i + 2} \frac{1 - x_{i+3} y_{i+3}}{1 - x_{i+1} y_{i+1}}.
\end{align*}
\end{prop}
The condition $1 - x_i y_i = 0$ causes vanishing denominators in the formula of $f$. This condition is equivalent to collinearity of the points $v_{i - 2}, v_i, v_{i+2}$.}

\new{These formulas were first published in \cite{MR2434454}, and a proof was given over $\R$ using $ab$-coordinates in \cite{MR2679816}; see also \cite[Section 2.3]{MR4460093}. Since the $ab$-coordinates are not well-defined over an arbitrary base field, we give a new field-independent proof of the formulas following the matrix refactorization argument of Izosimov \cite{MR4430020}. The remainder of this section builds up to the proof of Proposition \ref{prop_formulas}.
}

Recall the definitions of $\LeftSpace$ and $\RightSpace$ from Section \ref{sect_moduli}. First, we explain how to view $\RightSpace$ as a space of difference operators. We are working over an algebraically closed base field $k$. Let $k^\Z$ be the space of $\Z$-indexed sequences in $k$.
Let $\Shift \colon k^\Z \to k^\Z$
be the left shift operator, defined on each $\sigma \in k^\Z$ by
$$ (\Shift \sigma)_i = \sigma_{i+1}. $$
Each $s \in k^\Z$
defines a \emph{scalar operator} on $k^\Z$, via the rule
$$(s \sigma)_i = s_i \sigma_i.$$
A \emph{third-order difference operator} is a map $k^\Z \to k^\Z$ of the form
$$ a + b \Shift + c \Shift^2 + d \Shift^3,$$
where $a, b, c, d$ are scalar operators. Note that scalar operators generally do not commute with $\Shift$, since
\[
\bigl(\Sigma\circ s(\sigma)\bigr)_i = s_{i+1}\sigma_{i+1}
\quad\text{and}\quad
\bigl(s \circ\Sigma (\sigma)\bigr)_i = s_i\sigma_{i+1}.
\]

Any element of $\RightSpace$ can be identified with a third-order difference operator by reading off the sequences $a_i, b_i, c_i, d_i$. Specifically, the space~$\RightSpace$ is a subset of $(k^\Z)^4$, and we identify  each element $(a_i,b_i,c_i,d_i)_{i\in\Z}$ of~$\RightSpace$ with a third-order difference operator via
\[
(a_i,b_i,c_i,d_i)_{i\in\Z} \longleftrightarrow (a_i)+(b_i)\Sigma+(c_i)\Sigma^2+(d_i)\Sigma^3.
\]
Given any third-order difference operator $D = a + b \Shift + c \Shift^2 + d \Shift^3$, we define two associated third-order difference operators
$$\Dplus = a + c\Shift^2,$$
$$\Dminus = b\Shift + d\Shift^3.$$
Note that $D = \Dplus + \Dminus$.

We extend the action of a difference operator on $k^\Z$ coordinatewise to the set $(k^3)^\Z$ of sequences in $k^3$. Thus each third-order difference operator defines a map $\LeftSpace \to (k^3)^\Z$.

\begin{lemma}[\cite{MR4430020}] \label{lemma_pentagram_difference_op}
Let $u \in \Dom F$, let $v\in\LeftSpace$ be a lift of $u$, and let $D \in \RightSpace$ annihilate $v$. Then
\[
\Dminus v = F(u) \quad \text{as elements of $\TwParamNice$.}
\]
\end{lemma}

\begin{proof}
Observe that $(F(v))_i$ lies in the plane spanned by $v_i, v_{i + 2}$ and also in the plane spanned by $v_{i + 1}, v_{i + 3}$. We have 
$$\Dminus v = b\Sigma v + d\Sigma^3 v = -av - c\Sigma^2 v,$$
so $(\Dminus v)_i$ lies in the desired spans. Since $D$ is in $\RightSpace$, we know $(\Dminus v)_i$ is nonzero. The two planes intersect in a unique line, by the nondegeneracy condition on twisted $n$-gons. Thus, the image of $\Dminus v$ in $\TwParamNice$ is $F(u)$.
\end{proof}

The next theorem presents the pentagram map as a matrix refactorization.

\begin{thm}[\cite{MR4430020}] \label{thm_refac}
Suppose that $u \in \Dom F$.
Let $v$ be any lift of $u$ to $\LeftSpace$. Let $D \in \RightSpace$ annihilate $v$. Then there exists $\tilde{D} \in \RightSpace$ such that
\begin{equation} \label{eq_refac}
    \tildeDplus \Dminus = \tildeDminus \Dplus.
\end{equation}
Further, for each lift $\tilde{v} \in \LeftSpace$ of $F(u)$, there is a choice of $\tilde{D}$ such that
\begin{equation*}
    \tilde{D} (\tilde{v}) = 0.
\end{equation*}
\end{thm}

\begin{proof}
First we check that the factorization
\begin{equation}
\label{eqn:tildeDplusDminuseqtildeDminusDplus}
\tildeDplus \Dminus = \tildeDminus \Dplus
\end{equation}
is possible. We start by looking for a solution $\tilde{D}$ in the space of all $n$-periodic third-order difference operators (so some of the $a_i, b_i, c_i, d_i$ may be 0). The equation~\eqref{eqn:tildeDplusDminuseqtildeDminusDplus}
imposes $3n$ homogeneous linear conditions on $4n$ variables, so by linear algebra, there is a nontrivial solution.

Next we show that
$$ \tilde{D} (\tilde{v}) = 0.$$
Our initial choice of $\tilde{D}$ annihilates $\Dminus v$, by the following calculation:
\[
\begin{aligned}
    \tilde D v
    &= (\tildeDplus + \tildeDminus) \Dminus v 
    &&\text{since $\tilde D=\tildeDplus + \tildeDminus$,} \\
    &= \tildeDplus \Dminus v + \tildeDminus \Dminus v    \\
    &= \tildeDminus \Dplus v + \tildeDminus \Dminus v 
    &&\text{using \eqref{eqn:tildeDplusDminuseqtildeDminusDplus},}\\
    &= \tildeDminus (\Dplus + \Dminus) v \\
    &= \tildeDminus D v 
    &&\text{since $D=\Dplus+\Dminus$,}\\
    &= \tildeDminus 0
    &&\text{since $D v=0$ by assumption,}\\
    &= 0.
\end{aligned}
\]
We claim that $\tilde{D}$ belongs to $\RightSpace$. We note that $\Dminus$ has every consecutive 4-tuple in projective general position by the assumption that $u \in \Dom F$. Since $\tilde{D}(\Dminus v) = 0$ and $\tilde{D} \neq 0$, all the $a_i, b_i, c_i, d_i$ are nonzero, so $\tilde{D}$ is in $\RightSpace$.

By Lemma \ref{lemma_pentagram_difference_op}, we have that $\Dminus v$ and $\tilde{v}$ have the same image in $\TwParamNice$, so $\Dminus v$ and $\tilde{v}$ are the same up to rescaling vectors. Then we can rescale $\tilde{D}$ using the row-scalings $\rho_i$ of the action $\beta$ to produce a difference operator which annihilates $\tilde{v}$, and the modified $\tilde{D}$ still satisfies \eqref{eq_refac}.
\end{proof}

Now we calculate a formula for the pentagram map on $\Tw$. 

\begin{proof}[\new{Proof of Proposition \ref{prop_formulas}}]
Let $u, D, \tilde{D}$ be as in the statement of Theorem \ref{thm_refac}. Let $a, b, c, d$ be the sequences of coordinates of $D$, and let $\tilde{a}, \tilde{b}, \tilde{c}, \tilde{d}$ be the sequences of coordinates of $\tilde{D}$. Let $\sigma \in k^\Z$ be defined by
\[
\sigma_i = \begin{cases} 0 &\text{if $i\ne 6$,} \\ 1 &\text{if $i=6$.} \end{cases}
\]
Applying \eqref{eq_refac} to $\sigma$, we obtain the following system of equations:
\begin{align*}
    \tilde{c}_1 d_3 &= \tilde{d}_1 c_4, \\
    \tilde{c}_3 b_5 + \tilde{a}_3 d_3 &= \tilde{d}_3 a_6 + \tilde{b}_3 c_4, \\
    \tilde{a}_5 b_5 &= \tilde{b}_5 a_6.
\end{align*}
The same equations hold after a change in index, by symmetry. We use the appropriate shifts of the first and last equations to eliminate $\tilde{a}_i$ and $\tilde{d}_i$ from the middle equation, and rearrange to obtain
$$ \frac{\tilde{b}_3}{\tilde{c}_3} =
\frac{\tilde{b}_3}{\tilde{c}_3} = 
\frac{b_5 - \dfrac{d_5 a_6}{c_6}}{c_4 - \dfrac{a_4 d_3}{b_3}}.$$
Taking the reciprocal and multiplying by
$$ \frac{\tilde{a}_4}{\tilde{b}_4} = \frac{a_5}{b_4}$$
gives
$$ \frac{\tilde{a}_4 \tilde{c}_3}{ \tilde{b}_3 \tilde{b}_4 } = \frac{a_5 c_4}{b_4 b_5} \left( \frac{1 - \dfrac{a_4 d_3}{c_4 b_3}}{1 - \dfrac{d_5 a_6}{b_5 c_6}}. \right) $$
Then by Proposition \ref{prop_geoquo},
$$ f^*(x_5) = x_6 \frac{1 - x_5 y_5}{1 - x_7 y_7}.$$
Symmetry gives the formula for $x_i$ in general, and the formula for $y_i$ comes from a similar calculation.
\end{proof}

\section{The Lax Representation} \label{sect_lax}

We now derive a Lax representation for the pentagram map. The Lax representation supplies the invariant functions that define the invariant fibration of the pentagram map.

Informally, a Lax representation is an embedding of the domain of a dynamical system into a space of matrices such that the dynamical system is carried out via matrix conjugation.

\subsection{The idea}

The pentagram map is a discrete algebraic dynamical system
$$ f \colon \Tw \dashrightarrow \Tw,$$
so in our setting, the Lax representation is a space $\Mat_d(k)$ of $d \times d$ matrices over $k$ (for some $d \geq 1$) and a function 
$$T \colon \Tw \dashrightarrow \Mat_d(k)$$
such that, for any $v \in \Dom f$, the matrix $T(f(v))$ is conjugate (similar) to $T(v)$. Then the coefficients of the characteristic polynomial of $T(v)$ are functions on $\Tw$ which are invariant for $f$. One can do even better by identifying a function $P \colon \Tw \dashrightarrow \GL_d(k)$ which carries out the conjugation, that is, such that
$$ T \circ f = P T P^{-1}.$$
This construction supplies up to \new{$d$} invariant functions. More may be obtained from a Lax representation with \emph{spectral parameter.} This construction replaces $\Mat_d(k)$ with $\Mat_d(k(\zeta))$, where $\zeta$ is an indeterminate. Then the characteristic polynomial may be viewed as a bivariate polynomial, so there are more coefficients and thus more invariant functions. Of course, some of these invariant functions may be constants, and they may also be algebraically dependent, so even after a Lax representation with spectral parameter is found, some work is required to establish the existence of a fibration by low-dimensional subvarieties.

Formulas for $T$ and $P$ may be easily checked by carrying out the requisite matrix multiplication. But deriving the right formulas may be very difficult and requires some deeper sense of why the dynamical system is integrable. Soloviev found formulas for an algebraic Lax representation with spectral parameter for the pentagram map in $\Mat_3(\C)$ \cite{MR3161305}.  One hopes for a conceptually motivated proof, and \cite{MR3562255} has almost what we need, but the argument uses coordinates which are only well-defined over $\R$. Trying to push the same argument through over $\C$ results in invariants which are only well-defined up to a third root of unity. In this section, we derive the Lax representation using corner invariants as coordinates, so the argument works over any algebraically closed field $k$. 

Before formally going through the proof, we sketch the idea. A Lax representation (without spectral parameter) can be found by tracking the monodromy of a twisted $n$-gon $v$. The monodromy is a linear transformation, so it does not come with a preferred choice of coordinates. The Lax representation (without spectral parameter) describes the monodromy in a basis defined by the first few vertices of $v$. To introduce a spectral parameter, we observe a property of the pentagram map known as \emph{scaling invariance}. Namely, there is an action of $\Gm$ on $\Tw$, defined by
$$\zeta \cdot (x_1, \hdots, x_n, y_1, \hdots, y_n) = (\zeta^{-1} x_1, \hdots, \zeta^{-1} x_n, \zeta y_1, \hdots, \zeta y_n.)$$
Inspecting the formulas of Proposition \ref{prop_formulas}, we see that the pentagram map commutes with this action. This provides a deformation of the pentagram map, and following through the same construction for the deformation, we obtain matrices with the spectral parameter.

\begin{remark}
Another method of obtaining the $3 \times 3$ Lax representation is to go through the \new{``dual''} $n \times n$ Lax representation found by Izosimov in \cite{MR4430020}. \new{Indeed, writing down the eigenvectors of Izosimov's representation requires a recursive calculation equivalent to \eqref{eqn:defofTiprodLj} below. Thus a proof of algebraic integrability from the $n \times n$ Lax representation would ultimately lead through the same equations, resulting in the same spectral curve. However, there is a conceptual difference between the two approaches. The $3 \times 3$ Lax representation follows more directly from the geometric definition of the pentagram map on twisted polygons, but it does not explain the origins of scaling invariance. Izosimov's method allows for a far-reaching generalization of scaling invariance to a much wider class of generalized pentagram maps.}
\end{remark}

\subsection{The Lax representation via corner invariants}

We let $\zeta$ be a formal indeterminate. 

\begin{defn}
\label{defn:pzcewithsp}
A \emph{projective zero-curvature equation with spectral parameter} is the following data: for each $i \geq 1$, two functions
$$
L_{i} \colon \Dom f  \to \PP(\Mat_3(k(\zeta)))
\quad\text{and}\quad 
P_{i} \colon \Dom f \to \PGL_3(k(\zeta)),
$$
such that, for all $t \in \Z$,
$$L_{i, t+1} = P_{i+1, t} L_{i,t} {P_{i,t}}^{-1}.$$
\end{defn}

\begin{defn}
\label{defn:pLrwsp}
A \emph{projective Lax representation with spectral parameter} is, for each $i \geq 1$, two functions
$$T_{i}, P_{i} \colon \Dom f \to \PP(\Mat_3(k(\zeta)))$$
such that
$$T_{i, t+ 1} = P_{i,t} T_{i,t} {P_{i,t}}^{-1}.$$
In this situation, the functions $T_i$ are collectively called a \emph{projective Lax function with spectral parameter} and the function $P_i$ is called the \emph{associated function.} 
\end{defn}

Now we give explicit formulas for functions $L_i, P_i, T_i$ 
as in Definitions~\ref{defn:pzcewithsp} and~\ref{defn:pLrwsp}
that are associated to the pentagram map. For any $i \in \Z$, we define 
$$L_{i}, P_{i} \colon \Tw \dashrightarrow \PP(\Mat_3(k(\zeta)))$$ by
\begin{align}
\label{eqn:defofLiz}
L_i(\zeta) &= \begin{bmatrix}
1/x_{i+2} & -1/x_{i+2} & 0 \\
1/\zeta & 0 & -1/\zeta \\
y_{i+2} & 0 & 0
\end{bmatrix}, \\
\label{eqn:defofPiz}
P_i(\zeta) &= \begin{bmatrix}
1 - x_{i+2}y_{i+2} & 0 & -(1 - x_{i+2} y_{i+2}) \\
x_{i+1} y_{i+1}(1 - x_{i+2}y_{i+2}) & -(1 - x_{i+1}y_{i+1}) & -(1 - x_{i+2}y_{i+2}) \\
0 & \zeta y_{i+2} (1 - x_{i+3} y_{i+3}) & 0
\end{bmatrix}.
\end{align}

Note that $L_i$ and $P_i$ are $n$-periodic in $i$, since $x_i$ and $y_i$ are $n$-periodic in $i$.
We also introduce an optional time index $t \in \N_{\geq 0}$. We set
\[
L_{i, t} = L_i \circ f^{\circ t},\qquad
P_{i, t} = P_i \circ f^{\circ t},\qquad
T_{i, t} = T_i \circ f^{\circ t}.
\]

Also for any $i \in \Z$, we set
\begin{equation}
\label{eqn:defofTiprodLj}
T_i = L_{n-1+i} \cdots L_{i+1} L_i.
\end{equation}

\begin{thm} \cite[Theorem 2.2]{MR3161305} \label{thm_laxrep}
Suppose that $k$ is an algebraically closed field. Let~$L_i$,~$P_i$, and~$T_i$ be defined, respectively, by~\eqref{eqn:defofLiz},~\eqref{eqn:defofPiz}, and~\eqref{eqn:defofTiprodLj}.
Then we have
\begin{align}
    \label{eq_zero_curv}
    L_{i, t+1} &= P_{i+1, t} L_{i,t} {P_{i,t}}^{-1},\\
    \label{eq_lax}
    T_{i, t + 1} &= P_{i,t} T_{i,t} {P_{i,t}}^{-1}.
\end{align}
Thus the functions $L_i$ and $P_i$ define a projective zero-curvature equation with spectral parameter for the pentagram map, and $T_i$ is a projective Lax function with spectral parameter, with associated function $P_i$.
\end{thm}

\new{This Lax representation appeared as \cite[Theorem 2.2]{MR3161305}, where the proof is the brute-force calculation that the matrix equations \eqref{eq_zero_curv} and \eqref{eq_lax} follow from \eqref{eqn:defofLiz} and \eqref{eqn:defofPiz}. That approach requires foreknowledge of the hard-to-guess formulas \eqref{eqn:defofLiz} and \eqref{eqn:defofPiz}. Below, we give a new proof that leads naturally to these formulas. We first show that equations of the form of\eqref{eq_zero_curv} and \eqref{eq_lax} must hold for some matrices $L_i$ and $P_i$, then deduce the right definitions for $L_i$ and $P_i$ as a consequence.}

\new{This argument generalizes to other pentagram-like maps. In particular, it explains the origin of the complicated formulas in the Lax representation of the 3D pentagram map \cite[Theorem 6.3]{MR3118623}.}

\begin{proof}[Proof of Theorem \ref{thm_laxrep}]
We use the hypothesis that $k$ is algebraically closed to to define the moduli space $\Tw$ of twisted $n$-gons (Theorem \ref{thm_moduli}). Without loss of generality, set $t = 0$.

Suppose that $v \in \Dom f$. We set the notation $v' = f(v)$. 

Define a \emph{4-arc} to be a 4-tuple of points in general position. Given any pair of 4-arcs, there is a unique projective transformation mapping the first to the second. So, we can define projective transformations by specifying the image of a single 4-arc. We call this 4-arc the \emph{source} $4$-arc and we call its image the \emph{target} $4$-arc. Since $v \in \Dom f$, consecutive 4-tuples in $v'$ are 4-arcs.

For each $i \in \Z$, we now define three projective transformations 
$$\Lambda_i, \Lambda'_i, \Pi_i \colon \PP^2 \to \PP^2$$
by
\begin{align*}
{\Lambda_i} \colonequals (v_{i+1}, v_{i+2}, v_{i+3}, v_{i+4}) &\mapsto (v_i, v_{i+1}, v_{i+2}, v_{i+3}),\\
{\Lambda'_i} \colonequals (v'_{i+1}, v'_{i+2}, v'_{i+3}, v'_{i+4}) &\mapsto (v'_i, v'_{i+1}, v'_{i+2}, v'_{i+3}).\\
{\Pi_i} \colonequals (v'_i, v'_{i+1}, v'_{i+2}, v'_{i+3}) &\mapsto (v_i, v_{i+1}, v_{i+2}, v_{i+3}).
\end{align*}
Then
\begin{equation} \label{eqn_prelax}
    \Pi_{i+1}^{-1} \circ \Lambda_i^{-1} = {\Lambda'}_i^{-1} \circ \Pi_i^{-1},
\end{equation}
because both sides are projective transformations such that
$$(v_i, v_{i+1}, v_{i+2}, v_{i+3}) \mapsto (v'_{i+1}, v'_{i+2}, v'_{i+3}, v'_{i+4}).$$

We need to write the equality \eqref{eqn_prelax} in terms of matrices, which requires choosing bases. For each of the transformations $\Lambda_i, \Pi_i, \Lambda'_i$, we write a matrix $L_i, P_i, L'_i$ (respectively) for it, in the (unique) basis such that the given source 4-arc is located at
$$\bigl([1:0:0],\, [0:1:0],\, [0:0:1],\, [1:1:1]\bigr).$$
When we rewrite \eqref{eqn_prelax} as an equality in $\PP(\Mat_3(k))$, we get
\begin{equation} \label{eqn_lax_proof}
 L_{i}^{-1} P_{i+1}^{-1} = P_i^{-1} {L'}_i^{-1}.
\end{equation}
We rearrange \eqref{eqn_lax_proof} to be in the form of a projective zero-curvature equation for the pentagram map:
$$ {L'}_i P_i = P_{i+1} L_{i}.$$
The reversal of composition order is a consequence of the different choice of basis.

Now, we derive formulas for the matrices. For ease of notation, we just compute formulas for $L_1$ and $P_1$, and the rest follow by symmetry.

Define vectors
\[
e_1 = \begin{bmatrix}
1 \\
0 \\
0
\end{bmatrix}, \quad
e_2 = \begin{bmatrix}
0 \\
1 \\
0
\end{bmatrix}, \quad
e_3 = \begin{bmatrix}
0 \\
0 \\
1
\end{bmatrix}, \quad
e_4 = \begin{bmatrix}
1 \\
1 \\
1
\end{bmatrix}.
\]
We wish to write a matrix in the basis $e_1, e_2, e_3$ for the unique projective transformation that takes $e_i$ to a representative $e'_i$ of the corresponding point in the target 4-arc, for each $i = 1,2,3,4,$.
We can do this as follows.
\begin{enumerate}
    \item Calculate any choice of $e'_1, e'_2, e'_3$ in the chosen basis as a function of the coordinates $a_i, b_i, c_i, d_i$.
    \item Find values $\mu_1, \mu_2, \mu_3$ such that, up to rescaling,
    $$\mu_1 e'_1 + \mu_2 e'_2 + \mu_3 e'_3 = e'_4.$$
    If these are not obvious, we can find them by calculating some choice of $e'_4$ in the chosen basis and setting
    \[\begin{bmatrix}
    \mu_1 \\
    \mu_2 \\
    \mu_3
    \end{bmatrix} = 
    \begin{bmatrix}
    | & | & | \\
    e'_1 & e'_2 & e'_3 \\
    | & | & | \\
    \end{bmatrix}^{-1} e'_4.
    \]
    \item By direct check, the following matrix has the desired property:
    \[
    \begin{bmatrix}
    | & | & | \\
    \mu_1 e'_1 & \mu_2 e'_2 & \mu_3 e'_3 \\
    | & | & | \\
    \end{bmatrix}.
    \]
    \item Scale the matrix in Step 3 so that it is written in terms of the corner invariants $x_i, y_i$ instead of the $a_i, b_i, c_i, d_i$. We can ease this computation by choosing $a_1 = b_1 = c_1 = -d_1 = 1$, which is a choice consistent with our lift. We can also choose some, but not all, of the other variables to be $1$, by a rescaling.
\end{enumerate}

First we calculate ${L_1}^{-1}$ and invert it to obtain $L_1$. We can take
\[
e'_1 = \begin{bmatrix}
0 \\
1 \\
0 \\
\end{bmatrix}, \quad e'_2 = \begin{bmatrix}
0 \\
0 \\
1 \\
\end{bmatrix}, \quad e'_3 = \begin{bmatrix}
1 \\
1 \\
1 \\
\end{bmatrix}.
\]
Next we write $v_5$ in terms of $v_2, v_3, v_4$. Since
$$v_5 = - \frac{1}{d_2}(a_2 v_2 + b_2 v_3 + c_2 v_4),$$
we can take
$$\mu_1 = a_2, \quad \mu_2 = b_2, \quad \mu_3 = c_2.$$
Thus, up to scale,
\[
L_1^{-1} = \begin{bmatrix}
0 & 0 & c_2 \\
a_2 & 0 & c_2 \\
0 & b_2 & c_2 \\
\end{bmatrix}.
\]
We want to find a matrix projectively equivalent to this one that is written in terms of corner invariants. From Lemma \ref{lemma_corner_invariants}, we know that the entry with $a_2$ should be transformed to match $x_0$. So, dividing through by $b_2$ and using the known values $a_1 = b_1 = c_1 = 1$ and $d_1 = -1$, we obtain
\[
\begin{bmatrix}
0 & 0 & - \dfrac{c_1 c_2}{d_1 b_2} \\[4\jot]
\dfrac{a_2 c_1}{b_1 b_2} & 0 & - \dfrac{c_1 c_2}{d_1 b_2} \\[4\jot]
0 & 1 & - \dfrac{c_1 c_2}{d_1 b_2} \\
\end{bmatrix}.
\]
Thus
\[
{L_1}^{-1} = \begin{bmatrix}
0 & 0 & -1/y_3 \\
x_3 & 0 & -1/y_3 \\
0 & 1 & -1/y_3
\end{bmatrix}.
\]
We follow the same steps for ${P_1}^{-1}$. This time we get
\[
e'_1 = \begin{bmatrix}
b_0 \\
0 \\
d_0
\end{bmatrix}, \quad e'_2 = \begin{bmatrix}
1 \\
0 \\
1 \\
\end{bmatrix}, \quad e'_3 = \begin{bmatrix}
-c_2 \\
-c_2 + a_2 \\
-c_2
\end{bmatrix},
\]
\begin{align*}
    \mu_1 &= (a_3 d_2 - c_3 b_2), \\
    \mu_2 &= -(a_3 d_2 - c_3 b_2) b_0, \\
    \mu_3 &= - (d_0 - c_0 b_0) c_3.
\end{align*}

Thus, up to scale,
\[
P_1^{-1} =
\begin{bmatrix}
b_0 (a_3 d_2 - c_3 b_2) & -(a_3 d_2 - c_3 b_2) b_0 & c_2 c_3(d_0 - c_0 b_0) \\
0 & 0 & (a_2 - c_2)(d_0 - c_0 b_0)c_3 \\
d_0 (a_3 d_2 - c_3 b_2) & -(a_3 d_2 - c_3 b_2) b_0 & c_2 c_3 (d_0 - c_0 b_0)
\end{bmatrix}.
\]
Dividing through by $b_0 b_2 c_3$ and making the convenient choice $c_0 = 1$ gives us the desired form,
\[
P_1^{-1} = \begin{bmatrix}
-(1 - x_4 y_4) & 1 - x_4 y_4 & \frac{1}{y_3}(1 - x_2 y_2) \\
0 & 0 & -\frac{1}{y_3}(1 - x_2 y_2)(1 - x_3 y_3) \\
- x_2 y_2 (1 - x_4 y_4) & 1 - x_4 y_4 & \frac{1}{y_3} (1 - x_2 y_2)
\end{bmatrix}.
\]
Taking projective inverses and using symmetry, we deduce the formulas
\[
L_i = \begin{bmatrix}
1/x_{i+2} & -1/x_{i+2} & 0 \\
1 & 0 & -1 \\
y_{i+2} & 0 & 0
\end{bmatrix},
\]
\[
P_i = \begin{bmatrix}
1 - x_{i+2}y_{i+2} & 0 & -(1 - x_{i+2} y_{i+2}) \\
x_{i+1} y_{i+1}(1 - x_{i+2}y_{i+2}) & -(1 - x_{i+1}y_{i+1}) & -(1 - x_{i+2}y_{i+2}) \\
0 & y_{i+2} (1 - x_{i+3} y_{i+3}) & 0
\end{bmatrix}.
\]

The last step to prove the zero-curvature equation is to use the scaling invariance of the pentagram map to bring in a spectral parameter. The map
$$x_i \mapsto x_i/\zeta, \quad y_i \mapsto \zeta y_i,$$ 
commutes with the pentagram map. This map is an action of $\Gm$ on $\Tw$. The matrices appearing in \eqref{eqn_lax_proof} are functions of $\Tw$, but we can extend them to functions of $\Tw \times \Gm$ by setting $L_i(v, \zeta) = L_i(\zeta \cdot v)$ and $P_i(\zeta, v) = P_i(\zeta \cdot v)$.
The projective zero-curvature equation \eqref{eqn_lax_proof} still holds for these matrices because the pentagram map commutes with the $\Gm$-action. Suppressing the notational dependence on $v$, and rescaling the matrices by $\zeta$ or $1/\zeta$ as needed to get simpler formulas, we obtain the Lax representation with spectral parameter appearing in the theorem statement.

Finally, we use the zero-curvature equation to get the Lax equation. Since $P_i$ is $n$-periodic in $i$,
\begin{align*}
T_{i, t + 1} P_{i,t} &= L_{i + n - 1, t + 1} \hdots L_{i + 1, t + 1} L_{i, t+ 1} P_{i,t} \\
&= L_{i +  n - 1, t + 1} \hdots L_{i + 1, t + 1} P_{i + 1, t} L_{i, t} \\
& \vdots \\
&= P_{i + n, t} L_{i + n - 1, t} \hdots L_{i,t} \\
&= P_{i + n, t} T_{i, t} \\
&= P_{i, t} T_{i,t}.
\end{align*}
This conclude the proof of Theorem~\ref{thm_laxrep}.
\end{proof}

\begin{remark}
The Lax representation found by Soloviev differs from ours by some signs. The matrix in place of our $L_i$ is
\[
\begin{bmatrix}
1/x_{i+2} & -1/x_{i+2} & 0 \\
1/\zeta & 0 & 1/\zeta \\
-y_{i+2} & 0 & 0 \\
\end{bmatrix} = {\begin{bmatrix}
0 & 0 & -1/y_{i+2} \\
-x_{i+2} & 0 & -1/y_{i+2} \\
0 & \zeta & 1/y_{i+2}
\end{bmatrix}}^{-1},
\]
and the matrix in place of our $P_i(z)$ is
\[
\begin{bmatrix}
1 - x_{i+2}y_{i+2} & 0 & 1 - x_{i+2}y_{i+2} \\
x_{i+1}y_{i+1}(1 - x_{i+2}y_{i+2}) & 1 - x_{i+1}y_{i+1} & 1 - x_{i+2}y_{i+2} \\
0 & - \zeta y_{i+2}(1 - x_{i+3} y_{i+3}) & 0
\end{bmatrix}.
\]
This can be explained by a change of basis by $\diag(1,1,-1)$, so there is essentially no difference between the formulas.
\end{remark}

For an affine Lax representation, the natural next step is to deduce that the coefficients of the characteristic polynomial are invariants. We have to modify this step slightly because our Lax representation is projective.

Let $\Id_3$ denote the $3 \times 3$ identity matrix. Choose a nonzero representative $\hat{T}_0(\zeta)$ for $T_0(\zeta)$ in $\Mat_3(k[\zeta^{\pm 1}])$. We may view $\det (\lambda \Id_3 - \hat{T}_0(\zeta))$ as an element of $k[\Tw][\lambda, \zeta^{\pm 1}]$. The resulting expression is of the form
$$ \lambda^3 - \gamma_1(\zeta) \lambda^2 + \gamma_2(\zeta) \lambda - \gamma_3(\zeta).$$
The rescaling $\hat{T}_0 \mapsto l \hat{T}_0$ induces
the rescaling 
$$l \cdot (\gamma_1, \gamma_2, \gamma_3) = (l \gamma_1, l^2 \gamma_2, l^3 \gamma_3).$$
So, the coefficients of the $\gamma_i$ are functions on $\Tw$ such that the scaling class of the triple $(h_1, h_2, h_3)$ is invariant for the pentagram map. We can eliminate negative powers of $\zeta$ by multiplying through by $\zeta^n$. Then we can normalize by scaling so that $\gamma_1(\zeta)$ becomes monic. This puts the expression in the form
$$Q(\lambda, \zeta) \in k(\Tw)[\lambda, \zeta].$$
Since $Q$ is invariant under the pentagram map as a formal expression, the coefficients of $Q$ are rational functions on $\Tw$ which are invariant for the pentagram map. We would like to say which terms appear in $Q$. This may be done with an inductive calculation using the explicit formula for $L_i$; we refer to \cite[Prop. 5.3]{MR3102478} for the argument, and just state the end result.

\begin{cor} \label{cor_pentagram_invariants}
Set the notation $m = \floor{n/2}$. Then, for some $H_i \in k(\Tw)$ for $1 \leq i \leq 2m + 2$, we have
$$ Q(\lambda,\zeta) = \lambda^3 \zeta^n + \sum_{i = 0}^{m-1} H_{1+i} \lambda^2 \zeta^{n+i-m} - \lambda^2 \zeta^n + \sum_{i = 0}^m H_{m+i+1} \lambda \zeta^{m - i} - H_{2m+2}.$$
Let 
$$S = \A^{2m + 2}_k = \Spec k[h_1, \hdots, h_{2m+2}].$$ We define
$$H \colon \Tw \dashrightarrow S,$$ 
$$v \mapsto  (H_1(v), \hdots, H_{2m + 2}(v)).$$
The coefficients $H_1, \hdots, H_{2m+2}$ in the expression $Q(\lambda,\zeta)$ are rational functions on $\Tw$ that are invariants of the pentagram map. Thus the fibers of $H \colon \Tw \dashrightarrow S$ are invariant for the pentagram map.
\end{cor}

\section{The Geometry of the Spectral Curve} \label{sect_spectral_curve}

In Corollary \ref{cor_pentagram_invariants}, we described a bivariate polynomial $Q(\lambda,\zeta)$ with coefficients $H_1, \hdots, H_{2m+2}$ which are rational functions on $\Tw$ and invariant for the pentagram map. In this section, we view $Q(\lambda,\zeta)$ as the equation of a curve. 

We continue with the notation of Corollary \ref{cor_pentagram_invariants}. Set homogeneous coordinates $[X:Y:Z]$ on $\PP^2_{S}$. When we dehomogenize, we denote the corresponding remaining coordinates by the lowercase letters $x, y, z$. Thus if we dehomogenize by setting $Z = 1$, we write $x = X/Z$ and $y = Y/Z$.
(These are not to be confused with the corner invariants $x_i, y_j$.)

\begin{defn} \label{def_spec_curve}
For each $n \geq 4$, we define the \emph{$n$-th spectral curve} to be the relative curve $\Gamma \to S$ in $\PP^N_S$ cut out by the homogeneous polynomial
\begin{align*}
    Q(X,Y,Z) &= X^3 Y^n + \sum_{i = 0}^{m-1} h_{i+1} X^2 Y^{n+i-m} Z^{1+m-i} - X^2 Y^n Z \\
    &+ \sum_{i = 0}^m h_{m+i+1} X Y^{m - i} Z^{n-m+2+i} - h_{2m+2}Z^{3+n}.
\end{align*}
We may alternately view the relative curve $\Gamma \to S$ as a family of curves indexed by $S$.
\end{defn}

The spectral curve describes the eigenvalues in an algebraic 1-parameter family of linear maps, since $Q(\lambda,\zeta,1)$ equals the quantity $Q(\lambda, \zeta)$ described in Corollary \ref{cor_pentagram_invariants}.

Because the defining polynomial of the spectral curve is invariant for the pentagram map, we can think of the spectral curve itself as an invariant. In this section, we prove some geometric properties of the spectral curve in order to be able to describe its Jacobian. The results are collected in Proposition \ref{prop_sings} and Theorem \ref{thm_spec_curve_genus}.

The spectral curve was defined in (in a slightly different form) by Soloviev in \cite{MR3161305} to prove complex integrability over $\C$; in this section, we extend that argument to any algebraically closed field $k$, where $\charac k \neq 2$. The argument proceeds by desingularizing a family of curves and computing the genus of the generic fiber with the Riemann-Hurwitz formula. Directly computing the genus of the generic fiber is difficult, so we compute the genus of a carefully chosen special fiber, and then show that this special fiber has the same genus as the generic fiber. As discussed in Section \ref{sect_related_work}, our special fiber calculation fills a gap in \cite{MR3161305}.

Basic references for these techniques are as follows. For the resolution of curve singularities, and the basic theory of zeroes and poles of rational functions on singular varieties, see \cite[Ch. 1]{MR2289519}. For the basic theory of the Riemann-Hurwitz formula in arbitrary characteristic, see \cite[Ch. 4.2]{MR0463157} or \cite[Ch. 2]{MR817210}.

By \emph{curve}, we mean a projective, possibly singular algebraic variety of dimension 1 over an algebraically closed field. 

We prove various properties of the fibers of $\Gamma \to S$ which hold on a Zariski dense subset of $S$. Note that for any particular fiber, we may change coordinates to obtain a curve with affine equation of the form
$$ R(x,y) = x^3 y^n - \sum_{i = 0}^m J_i x^2 y^{n+i-m} + \sum_{i = 0}^m I_i x y^{m - i} - 1,$$
where $I_0, \hdots, I_m, J_0, \hdots, J_m \in k.$
This simpler form is unique up to a choice of third root of unity. We set 
$$S' = \Spec k[I_0, \hdots, I_m, J_0, \hdots, J_m].$$
For short, we write $k[I, J]$ to mean $k[I_0, \hdots, I_m, J_0, \hdots, J_m].$
We define a family of curves $\Gamma' \to S'$ in $\PP^2_{S'_n}$ by the homogeneous equation
$$R(X,Y,Z) = X^3 Y^n - \sum_{i=0}^m J_i X^2 Y^{n+i-m} Z^{1+m-i} + \sum_{i=0}^m I_i X Y^{m-i} Z^{n-m+2+i} - Z^{n+3} = 0.$$
A property of the fibers of $\Gamma \to S$ is said to hold \emph{generically} if, for all geometric points $s$ in some Zariski dense open subset of $S$, the property is true for $\Gamma_s$. Properties of curves which are stable under birational maps and which hold generically for fibers of $\Gamma'$ also hold generically for fibers of $\Gamma$.

We refer to the affine plane defined by $Z \neq 0$ as the \emph{main affine patch}.

For any point $s \in S'$, we consider the special fiber $\Gamma'_s \to \Spec k$. The geometry of the curve $\Gamma'_s$ varies with $s$; depending on the choice of $s$, the special fiber may be reducible, non-reduced, or have worse singularities and thus a lower genus than expected. The next theorem describes the geometry of $\Gamma'$.

\begin{prop} \label{prop_sings} \leavevmode
\begin{enumerate}
    \item For all $s \in S'$, the only points of $\Gamma'_s$ outside the main affine patch are $[1:0:0]$ and $[0:1:0]$, and these points are singular.
    \item Table \ref{table_resolution_of_singularities}
    describes a map $\Gamma'' \to \Gamma'$. For generic $s \in S'$, the corresponding map $\Gamma_s'' \to \Gamma'_s$ is a resolution of the singularities of $\Gamma'_s$ at $[1:0:0]$ and $[0:1:0]$ by point blowups.
    \item Let
    $$\nsrhmap  \colon\Gamma_s'' \to \PP^1$$
    be the map induced by
    $$\Gamma' \to \PP^1,$$
    $$[X:Y:Z] \to [Y:Z].$$
    The image and ramification indices of the geometric points above $[1:0:0]$ and $[0:1:0]$ are recorded in Table \ref{table_resolution_of_singularities}.
\end{enumerate}
\end{prop}

\begin{table}[p] 
\renewcommand{\arraystretch}{1.2}
\centering
\small
\begin{tabular}{c|c}
    \hline
    \multicolumn{2}{|c|}{$P = [1:0:0]$, \quad $n$ odd} \\
    \hline
     Coordinates & $y = \hat{y}^{m+1} \hat{z}, \quad
        z = \hat{y}^{m} \hat{z}$ \\
     Equation & $0 = \hat{y} - J_0 \hat{y} \hat{z} - \hdots - J_m \hat{y}^{m+1} \hat{z} + I_0 \hat{y}^m \hat{z}^2 + \hdots + I_m \hat{z}^2 - \hat{y}^m \hat{z}^3$ \\
     Points above $P$ & $O_2: (\hat{y}, \hat{z}) = (0,0)$ \\
     Image by $\nsrhmap$ & $[0:1]$ \\
     Ramification & 2 \\
    \hline
    \multicolumn{2}{|c|}{$P = [0:1:0]$, \quad $n$ odd} \\
    \hline
     Coordinates & $x = \hat{x}, \quad z = \hat{x} \hat{z}$. \\
     Equation & $0 = 1 - J_0 \hat{x}^m \hat{z}^{m+1} - \hdots - J_m \hat{z}
        + I_0 \hat{x}^{n-m}  \hat{z}^{2+n-m} + \hdots + I_m \hat{x}^n \hat{z}^{2+n} - \hat{x}^n \hat{z}^{3+n}$ \\
     Points above $P$ & $W_1: (\hat{x}, \hat{z}) = (0,1/J_m)$ \\
     Image by $\nsrhmap$ & $[1:0]$ \\
     Ramification & 1 \\
    \hline
     Coordinates & $x = \hat{x} \hat{z}^{m+1}, \quad z = \hat{z}$. \\
     Equation & $0 = \hat{x}^3 \hat{z}^m - J_0 \hat{x}^2 \hat{z}^m - \hdots - J_m \hat{x}^2
        + I_0 \hat{x} \hat{z}^{n-m} + \hdots + I_m \hat{x} \hat{z}^{n-m} - \hat{z}$ \\
     Points above $P$ & $W_2: (\hat{x}, \hat{z}) = (0,0)$ \\
     Image by $\nsrhmap$ & $[1:0]$ \\
     Ramification & 2 \\
    \hline
    \hline
    
    \multicolumn{2}{|c|}{$P = [1:0:0]$, \quad $n$ even} \\
    \hline
     Coordinates & $y = \hat{y}^{m+1} \hat{z}$, \quad $z = \hat{y}^{m} \hat{z}$ \\
     Equation & $0 = 1 - J_0 \hat{z} - \hdots - J_m \hat{y}^m \hat{z}
        + I_0 \hat{y}^m \hat{z}^2 + \hdots + I_m \hat{z}^2 - \hat{y}^m \hat{z}^3$ \\
     Points above $P$ & $O_2, O_3: (\hat{y}, \hat{z}) = (0,\hat{z}_0)$ for each of the two roots $\hat{z}_0$ of $I_m {\hat{z}_0}^2 - J_0 \hat{z}_0 + 1$. \\
     Image by $\nsrhmap$ & $[0:1]$ \\
     Ramification & 1 \\
    \hline
    \multicolumn{2}{|c|}{$P = [0:1:0]$, \quad $n$ even} \\
    \hline
     Coordinates & $x = \hat{x}, \quad z = \hat{x} \hat{z}$ \\
     Equation & $0 = 1 - J_0 \hat{x}^m \hat{z}^{m+1} - \hdots - J_m \hat{z} + I_0 \hat{x}^{n-m}  \hat{z}^{2+n-m} + \hdots + I_m \hat{x}^n \hat{z}^{2+n} - \hat{x}^n \hat{z}^{3+n}$ \\
     Points above $P$ & $W_1: (\hat{x}, \hat{z}) = (0,1/J_m)$ \\
     Image by $\nsrhmap$ & $[1:0]$ \\
     Ramification & 1 \\
    \hline
     Coordinates & $x = \hat{x} \hat{z}^{m+1}, \quad z = \hat{z}$. \\
     Equation & $0 = \hat{x}^3 \hat{z}^m - J_0 \hat{x}^2 \hat{z}^m - \hdots - J_m \hat{x}^2 + I_0 \hat{x} + \hdots + I_m \hat{x} \hat{z}^m - 1$ \\
     Points above $P$ & $W_2, W_3: (\hat{x}, \hat{z}) = (\hat{x}_0,0)$ for each of the two roots $\hat{x}_0$ of $J_m {\hat{x}_0}^2 - I_0 \hat{x}_0 + 1 = 0$. \\
     Image by $\nsrhmap$ & $[1:0]$ \\
     Ramification & 1 \\
\end{tabular}
\caption{A map $\Gamma'' \to \Gamma'$ generically resolving the singularities above $[1:0:0]$ and $[0:1:0]$, given in terms of coordinate changes and local equations. We also record the geometric points above the singularities, and their images and ramification by the map $\nsrhmap$.}
\label{table_resolution_of_singularities}
\end{table}

\begin{proof}
\par(1)\enspace To find the points that are outside the main affine patch, set $Z = 0$ and solve for $X$ and $Y$. The points $[1:0:0]$ and $[0:1:0]$ are singular by the Jacobian criterion for smoothness.
\par(2)\enspace 
We start by desingularizing $[1:0:0]$. Dehomogenizing by $X = 1$, we obtain the equation
        \begin{align*}
            R(1,y,z) = y^n &- J_0 y^{n-m} z^{m+1} - \hdots - J_m y^n z \\ &+ I_0 y^m z^{2+n-m} + \hdots + I_m z^{2+n} - z^{3+n}.
        \end{align*}
        In the $yz$-plane, the singularity is at $(0,0)$, so we have prepared it for blowing up. We treat $[0:1:0]$ similarly. We use the standard algorithm for blowing up a point singularity in the plane, and the results are as shown in Table \ref{table_resolution_of_singularities}. 
        In characteristic 2, we have used the separability of the polynomials $I_m \alpha^2 - J_0 \alpha + 1$ and $J_m \alpha^2 - I_0 \alpha + 1$ to justify that there are two roots.
\par(3)\enspace The images $\nsrhmap(P)$ and ramification indices of $\nsrhmap$ may be determined from the local equations in Table \ref{table_resolution_of_singularities}. For example, for the point above $[1:0:0]$ when $P$ is odd, the map $\nsrhmap$ is $\hat{y}$ in the given coordinates. The vanishing order of $\hat{y}$ is the length of the module $$\frac{k(I,J)[\hat{y}, \hat{z}]_{(\hat{y}, \hat{z})}}{\langle \hat{y}, \hat{y} - J_0 \hat{y}\hat{z} - \hdots - J_m \hat{y}^{n-m} \hat{z} + I_0 \hat{y}^m \hat{z}^2  + \hdots + I_m \hat{z}^2 - \hat{y}^m \hat{w}^3 \rangle} \cong    \frac{k(I,J)[\hat{z}]_{(\hat{z})}}{\langle I_m \hat{z}^2 \rangle}.$$
        Evidently this module has length 2, so $\hat{y}$ vanishes to order 2. So $\nsrhmap$ has a zero there, and the ramification index is 2.
\end{proof}

\begin{defn} \label{def_special_sections}
We give names to certain geometric points on $\Gamma''$. These may be thought of as sections of $\Gamma'' \to S$ in the case of points defined over $S$, or multisections in the case of points not defined over $S$.
\begin{itemize}
    \item \boxnodd: We name the preimages of $[0:1]$ by $\nsrhmap$ as follows. The point for which $Z \neq 0$ is $O_1$, and the point above $[1:0:0]$ is $O_2$. The points above $[0:1:0]$ are $W_1$ and $W_2$, as defined in Table \ref{table_resolution_of_singularities}. (Note that $O_1$ is not in Table \ref{table_resolution_of_singularities} because it does not lie above the line $Z = 0$.)
    \item \boxneven: There are three geometric points which are preimages of $\nsrhmap$. The point where $Z \neq 0$ is $O_1$, and the points above $[1:0:0]$ are $O_2, O_3$. The points $O_2$ and $O_3$ are not defined over the base $S'$. Fixing an algebraic closure of $k(I,J)$ allows us to give the names $O_2$ and $O_3$ to the points above $[1:0:0]$. The points $W_1, W_2, W_3$ are defined similarly according to Table \ref{table_resolution_of_singularities}, and again $W_2$ and $W_3$ are not defined over the base $S'$.
\end{itemize}
\end{defn}

\begin{thm} \label{thm_spec_curve_genus}
Assume $\charac k \neq 2$. For generic $s \in S$, the curve $\Gamma_s$ over $k$ is integral, and has geometric genus
\[
g(\Gamma_s) = \begin{cases}
n - 1, & n \text{ odd,}\\
n - 2, & n \text{ even.}
\end{cases}
\]
\end{thm}

We prove the theorem for $\Gamma'_s$, and the result for $\Gamma_s$ follows immediately because geometric genus is a birational invariant. We split the proof of Theorem \ref{thm_spec_curve_genus} into lemmas, which make up the remainder of this section.

The idea of the proof is to compute the genus of a special fiber. We have to choose the special fiber judiciously, because (roughly speaking) if the fiber we choose has worse singularities than the generic fiber, then the genus will be lower.
\pagebreak[2] 
\begin{defn} \label{def_good_fiber}
We call $\Gamma'_s$ a \emph{good fiber} if it satisfies all of the following conditions:
\begin{enumerate}
    \item The $k$-scheme $\Gamma'_s$ is integral.
    \item The singularities at $[1:0:0]$ and $[0:1:0]$ of $\Gamma'_s$ are resolved by $\Gamma''_s \to \Gamma'_s$. Equivalently, the following nondegeneracy conditions hold:
    $$I_m \neq 0, \quad J_m \neq 0, \quad {J_0}^2 - 4I_m \neq 0, \quad {I_0}^2 - 4 J_m \neq 0.$$
    \item The curve $\Gamma''_s$ is nonsingular in the main affine patch. Equivalently, there are no singularities of $\Gamma'_s$ in the main affine patch.
\end{enumerate}
\end{defn}

The next lemma shows that the generic curve's genus can be computed by looking at a good fiber, because a good fiber's singularities are as mild as possible for the family $\Gamma'$.

\begin{lemma} \label{lemma_good_fiber_suffices}
Suppose that a good fiber $\Gamma'_s$ exists. Then generically, the fibers of $\Gamma'$ have the same genus as $\Gamma'_s$.
\end{lemma}

\begin{proof}
    Integrality is an open condition, because the locus of reducible polynomials of a fixed degree is an algebraic subset of the space of all polynomials of that degree. Thus, since $\Gamma'_s$ is integral, so is the generic fiber of $\Gamma'$.
    
    We claim that for generic $s'$, we have
    $$ g(\Gamma'_{s}) \geq g(\Gamma'_{s'}).$$
    
    Since geometric genus is a birational invariant,
    it suffices to show
    $$ g(\Gamma''_{s}) \geq g(\Gamma''_{s'}).$$
    For generic $s'$, the singularities of $\Gamma'_{s'}$ at $[1:0:0]$ and $[0:1:0]$ have the structure described in Table \ref{table_resolution_of_singularities}.
    Denote arithmetic genus by $\arithgenus$. 
    By the good fiber hypothesis, $g(\Gamma''_s) = \arithgenus(\Gamma''_s)$.
    Observe that $\Gamma'_s$ and $\Gamma'_{s'}$ are plane curves with the same degree, hence the same arithmetic genus. 
    Combining this with the fact that the singularities of $\Gamma'_s$ and $\Gamma'_{s'}$ at $[1:0:0]$ and $[0:1:0]$ have the same structure, we find
    $$\arithgenus(\Gamma''_s) =\arithgenus(\Gamma''_{s'}).$$
    For any curve, the arithmetic genus is at least the geometric genus, so
    $$\arithgenus(\Gamma''_{s'}) \geq g(\Gamma''_{s'}).$$
    Putting it all together proves the claim:
    $$g(\Gamma'_s) = g(\Gamma''_s) = \arithgenus(\Gamma''_s) \geq \arithgenus(\Gamma''_{s'}) \geq g(\Gamma''_{s'}) .$$
    Since all the curves in $\Gamma'$ have the same degree, the family $\Gamma'$ is flat, so the function $S' \to \Z_{\geq 0}$, $s \mapsto g(\Gamma'_s)$ is lower-semicontinuous in the Zariski topology. Therefore, the genus for generic $s' \in S$ must be exactly $g(\Gamma'_s)$.
\end{proof}

The next lemma describes the essential idea of the strategy for calculating the genus of a good fiber. To state it, we need more notation.

Consider the map $\nsrhmap \colon\Gamma''_s \to \PP^1$ from Proposition \ref{prop_sings}. Let $\Omega$ be the sheaf of relative differentials on $\Gamma''_s$, where the structure map is $\nsrhmap$. 

For any point $P \in \Gamma''_s$, we denote the length of $\Omega$ at $P$ by $\len( \Omega_P )$. For short, we also write $\omega(P) = \len(\Omega_P).$

We denote the vanishing order of a function $\Phi \in K(\Gamma''_s)$ at $P$ by $\ord_P \Phi$. If $\Phi$ has a pole at $P$, then $\ord_P \Phi$ is negative.

We use the notation
 $$R_x = \frac{\partial R}{\partial x}, \quad R_y = \frac{ \partial R }{\partial y}.$$
Note that the functions $R_x, R_y$ are computed by dehomogenizing $R$ so that $Z = 1.$

We also set the notation $P_\infty$ for the set of points of $\Gamma''_s$ above the line $Z = 0$. Thus,
\begin{itemize}
    \item \boxnodd: \qquad
    $P_\infty = \{O_2, W_1,W_2\}.$
    \item \boxneven:\qquad
    $P_\infty = \{O_2, O_3, W_1, W_2, W_3\}.$
\end{itemize}

\begin{lemma} \label{lemma_good_fiber_strategy}
If $\Gamma''_s$ is a good fiber, then
$$g(\Gamma''_s) = \frac{1}{2} \biggl( -4 - \sum_{P \in P_\infty} \ord_P R_x + \sum_{P \in P_\infty} \omega(P)\biggr).$$
\end{lemma}

\begin{proof}
 Because $\Gamma''_s$ is a good fiber, it is a nonsingular, integral curve. Further, since $\Gamma''_s$ is a good fiber, we have $I_m \neq 0$. Working in dehomogenized coordinates $x,y$, we see by inspection that $y$ is a uniformizer at
 $$(x,y) = (1/I_m, 0) \in \Gamma'_s.$$
 Therefore $\nsrhmap$ is separable. So the Riemann-Hurwitz formula applies to $\nsrhmap$. Since $\nsrhmap$ is degree 3, the Riemann-Hurwitz formula states
 $$ 2g(\Gamma''_s) - 2 = (-2) \cdot 3 + \sum_{P \in \Gamma''_s} \omega(P).$$
 Breaking apart the sum and rearranging this formula gives
 $$ g(\Gamma''_s) = \frac{1}{2}\biggl(-4 + \sum_{P \not\in P_\infty} \omega(P) + \sum_{P \in P_\infty} \omega(P)\biggr).$$
 We compute the first sum another way. Suppose that $P \not\in P_\infty$. Then $P$ is in the main affine patch. Working in dehomogenized coordinates $x, y$, we can compute the module of relative differentials at $P$ explicitly.
 \begin{align*}
     \Omega_P = \left(  \frac{ k \langle dx, dy \rangle }{ dy, R_x dx + R_y dy } \right)_P 
     \cong \left(  \frac{ k \langle dx \rangle }{ R_x dx} \right)_P.
 \end{align*}
 Thus
 $$\omega(P) = \ord_P( R_x ).$$
 Then, since the zeros and poles of a rational function have total vanishing order 0, we get
 $$\sum_{P \not\in P_\infty} \omega(P) = -\sum_{P \in P_\infty} \ord_P R_x.$$
 This completes the proof of Lemma~\ref{lemma_good_fiber_strategy}.
\end{proof}

Now we come to the last step, which takes some luck. In each characteristic, we need to find a good fiber and compute the vanishing order $\ord_P R_x$ and $\omega(P)$ for each point $P \in P_\infty$. We split the job into cases depending on the characteristic and $n$.

For the remainder of the section, we set notation for the quantity which we claim is the expected genus,
\[
g(n) = \begin{cases}
n - 1 & n \text{ odd,} \\
n - 2 & n \text{ even.} \\
\end{cases}
\]

\begin{lemma} \label{lemma_case_not_2n}
Suppose $\charac k \neq 2$ and $\charac k \nmid n$. The curve $\Gamma'_s$ cut out by
\begin{equation} \label{eq_tame_case}
    R(x,y) = x^3 y^n - x^2 y^n - x - 1 = 0
\end{equation}
is a good fiber, and its genus is $g(n)$. This fiber corresponds to $s \in S'$ defined by the coordinates $I_m = -1, J_m = 1$, and all other $I_i, J_i = 0$. 
\end{lemma}

\begin{proof}
    First we check the three conditions for  $\Gamma'_s$ to be a good fiber. 
    \begin{enumerate}
    \item  Since $\charac k \neq 2$, the Eisenstein criterion applies with reference to the ideal $(x + 1)$, so $\Gamma'_s$ is integral.
    \item Since $\charac k \neq 2$, the quantities $I_m, J_m, {I_0}^2 - 4J_m, {J_0}^2 - 4 I_m$ are all nonzero.
    \item We claim that there are no singularities in the main affine patch. This follows from the Jacobian criterion for smoothness. We need to show there are no simultaneous solutions $(x,y)$ of
    \[
\begin{array}{rclcl}
  R(x,y) & = & x^3 y^n - x^2 y^n - x - 1 & = & 0, \\
  R_x (x,y) & = & 3x^2 y^n - 2x y^n - 1 & = & 0, \\
  R_y (x,y) & = & n(x^3 - x^2) y^{n-1} & = & 0. \\
\end{array}
\]
    Since by assumption $n \neq 0$ in our characteristic, we deduce from $R_y = 0$ that either $z = 0, k = 0$, or $k = 1$. Looking at $R_x$, only the last case is possible, and in that case $y^n = 1$. But plugging this information into $R$, we have $-2 = 0$, which is false since $\charac k \neq 2$.
    \end{enumerate}
    Next, we compute the quantities $\omega(P)$ and $\ord_P R_x$ for each $P \in P_\infty$. We list the results of the calculation in Table \ref{table_genus_computation_case_not_2n}.
\begin{table}[h] 
\renewcommand{\arraystretch}{1.2}
\centering
\small
\begin{tabular}{c|c|c}
    \hline
    \multicolumn{3}{|c|}{$n$ odd} \\
    \hline
     $P$ & $\omega(P)$ & $\ord_P R_x$ \\
     \hline
     $O_2$ & 1 & 0\\
     $W_1$ & 0 & $-n$\\
     $W_2$ & 1 & $-n$\\
     \end{tabular}
     \quad \quad \quad
     \begin{tabular}{c|c|c}
     \hline
        \multicolumn{3}{|c|}{$n$ even} \\
    \hline
     $P$ & $\omega(P)$ & $\ord_P R_x$ \\
     \hline
     $O_2$ & 0 & 0\\
     $O_3$ & 0 & 0\\
     $W_1$ & 0 & $-n$\\
     $W_2$ & 0 & $-n/2$\\
     $W_3$ & 0 & $-n/2$\\
\end{tabular}
\caption{Calculation of the genus of $\Gamma'_s.$}
\label{table_genus_computation_case_not_2n}
\end{table}
To compute these quantities, first we compute $\omega(P)$ for each $P \in P_\infty$. We can inspect the equations in Table \ref{table_resolution_of_singularities} to deduce the ramification index $e(P)$ at each $P \in P_\infty$, and we find that $e(P) \leq 2$ for each $P$. So, by the assumption that $\charac k \neq 2$, the ramification at each $P$ is tame. So $\omega(P) = e(P) - 1.$
    
Next, we compute $\ord_P R_x$ for each $P \in P_\infty$. 
\par\textbullet\enspace Let $P = O_2$. First we calculate
        $$R_x = 3x^2 y^n - 2 xy^n - 1.$$
        We homogenize, then dehomogenize by $X = 1$. We obtain the local equation
        $$R_x = \frac{1}{z^{2+n}} (3y^n - 2 y^{n- m}z^{1+m} - z^{2+n}).$$
        We blow up $(y,z) = (0,0)$ according to the formula in Table \ref{table_resolution_of_singularities}. We compute the equation of $R_x$ in the coordinates $\hat{y}, \hat{z}$. We split into cases for $n$ odd and even. Suppose that $n$ is odd; then
        $$R_x = \frac{3\hat{y} - 2 \hat{y}^{m+1} \hat{z} - \hat{z}^2}{\hat{z}^2}.$$
        Because $\ord_P$ is a valuation, if one of these terms has smaller order than all the others, then $\ord_P R_x$ is computed by the order of that term. But in this case there are two terms of least order: $$\ord_P(\hat{y}/\hat{z}) = \ord_P(1) = 0.$$
        We eliminate one term by adding the appropriate multiple of $R$; this leaves one term,~$2$, of minimal order 0. Then, since $\charac k \neq 2$, the vanishing order at $P$ is 0. 
        
        On the other hand, suppose that $n$ is even; then
        $$R_x = \frac{3 - 2 \hat{y}^m \hat{z} - \hat{z}^2}{\hat{z}^2}.$$
        The coordinates of $O_2, O_3$ are $(\hat{y}, \hat{z}) = (0, \pm 1)$, so we can see directly by plugging in values that, since $2 \neq 0$, the function $R_x$ has order 0 at $O_2, O_3$.
    \par\textbullet\enspace
        We compute the order at $P = W_1$. The equation of $R_x$ in $\hat{x}, \hat{z}$ is
        $$R_x = \frac{3 - 2\hat{z} - \hat{x}^n}{\hat{x}^n \hat{z}^{2+n}}.$$
        The numerator at $W_1 = (0,1)$ is 1. The order of the denominator is computed by examining the local equation for $R$. We find that $\hat{z}$ has order 0, and $\hat{x}$ has order 1, so $\ord_P R_x = -n.$
    \par\textbullet\enspace The argument for $W_2, W_3$ uses the same techniques, so we omit it.
    
    Finally, we plug the quantities in Table \ref{table_genus_computation_case_not_2n} into Lemma \ref{lemma_good_fiber_strategy}.
\end{proof}

\begin{lemma} \label{lemma_case_n}
Suppose $\charac k \neq 2$ and $\charac k \mid n$. The curve $\Gamma'_s$ cut out by
\begin{equation} \label{eq_wild_case}
R(x,y) = x^3 y^n + x^2 y^{n-1} - x^2 y^n + x - 1
\end{equation}
is a good fiber, and its genus is $g(n)$. This fiber corresponds to $s \in S'$ defined by the coordinates $I_m = 1, \; J_{m -1} = -1, \; J_m = 1$, and all other $I_i, J_i = 0$. 
\end{lemma}

\begin{proof}
The structure of the proof is identical to that of Lemma \ref{lemma_case_not_2n}, but the computations are different. First we check the three conditions for $\Gamma'_s$ to be a good fiber.
\begin{enumerate}
    \item We claim $\Gamma'_s$ is integral. If it were not, then $R(x,y)$ would factor in $k[x,y]$, and since $R$ is cubic in $x$, there would be a factor linear in $x$. Since $\Gamma'_s$ contains no points of the form $y = 0$ in the main affine patch, this line would be of the form $x = c$ for some constant $c \in k$. But this cannot happen, since if $R(c,y)$ were identically 0, then examining coefficients, we would have both $c = 0$ and $c = 1.$
    \item Since $\charac k \neq 2$, the quantities $I_m, J_m, {I_0}^2 - 4J_m, {J_0}^2 - 4 I_m$ are all nonzero.
    \item We claim that $\Gamma'_s$ is nonsingular in the main affine patch. To see this, observe that a singularity $(x,y)$ would satisfy
    $$R_y (x,y) = -x^2 y^{n-1} = 0,$$ 
    since $\charac k \mid n$. So if $(x, y)$ is a singularity, then $x = 0$ or $y = 0$. But for such points, we have 
    $$R_x(x,y) = 3x^2 y^n - 2x y^{n-1} - 2x y^n + 1 = 1 \neq 0.$$
\end{enumerate}
The values of $\omega(P)$ and $\ord_P R_x$ coincide with the numbers in Table \ref{table_genus_computation_case_not_2n}. The computation is nearly identical, so we omit it.
\end{proof}
\new{
\begin{remark} \label{rem_source_of_polys}
We briefly explain how the polynomials \eqref{eq_tame_case} and \eqref{eq_wild_case} were found. The first is
$$R(x,y) = x^3 y^n - x^2 y^n - x - 1 = 0.$$
We tried various polynomials that had few terms and this was the simplest one that we found that had small genus. But this polynomial is not suitable for the case when the the characteristic of $k$ divides $n$, because then the $y$-derivative $R_y(x,y)$ is identically 0, and the Jacobian criterion reveals extra singularities in the main affine patch. Thus for this case we use the polynomial 
$$R(x,y) = x^3 y^n + x^2 y^{n-1} - x^2 y^n + x - 1.$$
Characteristic 2 is more difficult than the others because a hypothesis in the definition of good fiber (Definition \ref{def_good_fiber}) reduces to $$I_0, I_m, J_0, J_m \neq 0$$
in characteristic 2. Thus a good fiber needs to have at least 6 monomials in its defining equation, which makes calculating the singularities with the Jacobian criterion very complicated. This is the reason for the hypothesis on characteristic in Theorem \ref{thm_main_1}.
\end{remark}
}
Finally, we collect these results to prove Theorem \ref{thm_spec_curve_genus}.

\begin{proof}[Proof of Theorem \ref{thm_spec_curve_genus}]
Combine Lemma \ref{lemma_good_fiber_suffices}, Lemma \ref{lemma_case_not_2n}, Lemma \ref{lemma_case_n}.
\end{proof}

\section{The Spectral Transform} \label{sect_spectral_transform}

We now describe the construction of the \emph{direct spectral transform}, the birational map 
$$\delta \colon \Tw \dashrightarrow \abfam$$ of Theorem \ref{thm_main_1}. The idea is that a twisted $n$-gon $v$ can be reconstructed from the Lax matrices $T_i(v, \zeta)$, and matrices can be reconstructed from their eigenvalues and eigenvectors, at least generically. The eigenvalues correspond to the three points of the spectral curve $\Gamma_{H(v)}$ above $\zeta$. Each eigenvalue has an associated eigenvector (up to scale), and these fit together into a line bundle on $\Gamma_{H(v)}$.

In fact, Soloviev's argument over $\C$ in \cite{MR3161305} goes over to our setting with only cosmetic changes, so we just formulate the statements we need and appeal to \cite{MR3161305} for the proofs.

 Let $\Gamma \to S$ be the spectral curve (Definition \ref{def_spec_curve}). The map $\Gamma \to S$ is projective, flat, and finitely presented, since it is a family of projective plane curves of the same degree. By Theorem \ref{thm_spec_curve_genus}, the fibers are integral schemes of dimension 1 over an algebraically closed field. Then by \cite[Theorem 8.2.1]{MR1045822}, the relative Picard scheme $\Pic_{\Gamma/S}$ exists and has the structure of both a $k$-variety and an $S$-scheme.
    
zThe map $H \colon \Tw \dashrightarrow S$ gives a dense, Zariski open subset $\Tw^\circ \subset \Tw$ the structure of an $S$-scheme. For each $v \in \Tw^\circ$, the special fiber $\Gamma_{H(v)}$ is a degree 3 curve. The points $(\lambda, \zeta)$ of $\Gamma_{H(v)}$ in the main affine patch parametrize eigenvalues $\lambda$ of $T_0(v, \zeta)$. The map $\zeta \colon \Gamma_{H(v)} \to \PP^1$ is generically 3-to-1, so there is a unique $\lambda$-eigenvector $\psi$ (up to scale) of $T_0(v,\zeta)$ associated to a generically chosen point $(\lambda, \zeta)$ on $\Gamma_{H(v)}$. The eigenvector's coordinates are rational functions of $\lambda, \zeta$ and $v$. Because $\psi$ varies algebraically in $\lambda$ and $\zeta$, and $\Gamma_{H(v)}$ is a projective curve, we can extend $\psi$ in a unique way to a line bundle on the normalization. Since $\Tw^\circ$ is an $S$-scheme, the family of line bundles $\psi$ is represented by a $\Tw^\circ$-point on $\Pic_{\Gamma/S}$. So there is a corresponding map
    \new{$$\Psi \colon \Tw^\circ \to \Pic_{\Gamma/S}.$$}

\begin{lemma}
For any algebraically closed field $k$ and for generic $v$ in $\Tw^\circ$, the line bundle $\psi_{H(v)}$ has degree $g(n) + 2$.
\end{lemma}
\begin{proof}
Over $\C$, this follows from the fact that the ramification points of $\nsrhmap$ on $\Gamma$ are generically simple; a proof is given in \cite[Lemma 3.4]{MR3161305}. We extend this to any characteristic by observing that the degree of a line bundle defined over $\Z$ is preserved by base change. More formally, let $S_\Z = \A^{2m+2}_\Z$, and  consider the subscheme $\Gamma$ of $\PP^2_{S_\Z}$ cut out by the affine equation $Q(\lambda,\zeta) = 0$ defining the spectral curve. Define $(\Tw^\circ)_\Z$ similarly. The scheme $\Pic_{\Gamma_{\Z} / S_{\Z}}$ exists, and the eigenvector bundle is a map from $\Gm^{2n}$ to $\Pic_{\Gamma_{\Z}/S_{\Z}}$. Since the eigenvector bundle has degree $g(n) + 2$ for $\C$-points, it must send the generic point of $(\Tw^\circ)_\Z$ into $\Pic_{\Gamma_{\Z}/S_{\Z}}^{g(n)+2}$.
\end{proof}

\begin{defn} \label{def_spectral_transform}
For each $n \geq 4$, we define a relative abelian variety $\abfam$, the \emph{spectral data}, and a rational map
$$\delta \colon \Tw \dashrightarrow \abfam,$$
the \emph{spectral transform}. The definitions of $\abfam$ and $\delta$ depend on the parity of $n$, as follows.

\begin{itemize}
    \item \boxnodd: Let $\abfam = \Pic_{\Gamma/S}^{n+1}$, and let $\delta = \new{\Psi}$. The points $O_1, O_2, W_1, W_2$ were defined as sections of $\Gamma' \to S'$; see Definition \ref{def_special_sections}. By abuse of notation we let $O_1, O_2, W_1, W_2$ denote the corresponding sections of $\Gamma \to S$. The morphism $\delta$ on $\Tw^\circ$ extends to a rational map
    $$\delta \colon \Tw \dashrightarrow \abfam.$$
    \item \boxneven: In this case, we need to mark the points $O_2, O_3, W_2, W_3$ appearing in Table \ref{table_resolution_of_singularities}. Extend the base $S$ by adjoining a root $\hat{x}_0$ of ${\hat{x}_0}^2 - h_1 {\hat{x}_0} + h_{n}$ and a root $\hat{z}_0$ of $h_{n+1} {\hat{z}_0}^2 - h_{m+1} \hat{z}_0 + 1$. (These correspond to the roots of the polynomials $J_m {\hat{z}_0}^2 - I_0 \hat{z}_0 + 1$ and $I_m {\hat{x}_0}^2 - J_0 \hat{x}_0 + 1$ appearing in Table \ref{table_resolution_of_singularities}.) This defines a generically $4$-to-$1$
    cover $S^{\marked} \to S$, and by pullback, we define a family $\Gamma^{\marked} \to S^{\marked}$. The members of this family are spectral curves with a marking of the points $O_2, O_3$ and $W_2, W_3$. We define $$\abfam = \Pic_{\Gamma^{\marked}/S^{\marked}}^{n}.$$
    The latter relative Picard scheme exists by the same condition we used for $\Pic_{\Gamma/S}.$
    Recall that $x_i, y_i$ denote the corner invariants, which are coordinates on $\Tw$. The map $H$ factors through $H^{\marked} \colon \Tw^\circ \to S^{\marked}$, by choosing
$$ \hat{x}_0 = \prod_{i=0}^q x_{2i}, \quad \hat{z}_0 = \prod_{i=0}^q y_{2i}.$$
So $\Tw^\circ$ has the structure of an $S^{\marked}$-scheme. Define
$$\delta \colon \Tw^\circ \to \Pic_{\Gamma^{\marked}/S^{\marked}}$$
to be the map induced by pulling back $\new{\Psi}$. Then $\delta$ defines a rational map
$$ \delta \colon \Tw \dashrightarrow \Pic_{\Gamma^{\marked} / S^{\marked}}.$$
\end{itemize}
\end{defn}

\begin{thm} \label{thm_spectral_transform}
Assume that $\charac k \neq 2$. Then the spectral transform $\delta \colon \Tw \dashrightarrow \abfam$ of Definition \ref{def_spectral_transform} is a birational map.
\end{thm}

\begin{proof}
The construction of the rational inverse of $\delta$, the \emph{inverse spectral transform}, is carried out over $\C$ in \cite[Section 3.2]{MR3161305}. It runs for several pages and works almost verbatim for our setting. The structure of the argument over $\C$ is as follows. First, describe the divisor of $\psi_{H(v)}$ on $\Gamma_{H(v)}$, for generic $v$, by an explicit calculation of the vanishing order for each point in $P_\infty$. Second, the genus of the spectral curve (for generic $s \in S$) is $g(n) + 2$, so the Riemann-Roch theorem can be used to show that there is only one such divisor in each linear equivalence class. Third, show that a Lax matrix is determined by the data of its eigenvalues and eigenvector divisor.

Now we adapt the argument to the field $k$. For the first step, the divisor calculations are the same \new{for any algebraically closed field $k$, by inspection of the (many) formulas in \cite[Appendix]{MR3161305}. Note that the divisions by 2 that appear in this argument are not used in the divisor calculations, but rather to construct the universal symplectic form, so they are not relevant to the present argument.} The limits which appear in the calculation are just the leading terms appearing in certain Laurent expansions, so these calculations are all algebraic. We need to work with corner invariants $x_i, y_i$ instead of $ab$-coordinates, but we get the same results.

For the second step, the genus of the generic spectral curve is $g(n)$ if $\charac k \neq 2$, by Theorem \ref{thm_spec_curve_genus}.

The third step, the reconstruction of the Lax matrix, is algebraic in nature and makes sense for any algebraically closed base field.
\end{proof}

While the genus of the spectral curve $\Gamma_s$ for generic $s \in S$ is $g(n) + 2$ by Theorem \ref{thm_spec_curve_genus}, it does not follow that the genus of $\Gamma_{H(v)}$ for generic $v$ is $g(n) + 2$. In fact, there are loci in $S$ that correspond to spectral curves of lower genus, such as the image of $H$ by the subset of $\Tw$ consisting of closed polygons. One way to resolve this issue is to show that the pentagram invariants (the coordinates of $H$) are algebraically independent, so $H$ is dominant; this was done over $\C$ in \cite{MR2434454}. The argument we use here constructs
an explicit birational inverse for $\delta$. Comparing the dimensions of $\Tw^\circ$ and $\abfam$, we obtain algebraic independence of the pentagram invariants as a corollary.

\begin{cor}
The invariants of $f$ are algebraically independent over any field $k$ for which $\charac k \neq 2$.
\end{cor}

Finally, we describe the effect of the pentagram map $f$ after birationally identifying $\Tw$ with $\abfam$ via $\delta.$

\begin{thm} \label{thm_translation}
   Let $k$ be an algebraically closed field, and assume $\charac k \neq 2$.
   \begin{itemize}
       \item \boxnodd: Let 
       $$\tau \colon \abfam \to \abfam$$
       denote translation by the section $[-O_1 + W_2] \in \Pic_{\Gamma/S}^0$. Then as rational maps,
       $$f = \delta^{-1} \circ \tau \circ \delta.$$
       \item \boxneven: Let
       $$\tau \colon \abfam \to \abfam$$
       denote translation by the section $[-O_1 + W_2] \in \Pic_{\Gamma^{\marked} /S^{\marked}}$. 
       Let $\iota$ be the map on $\abfam$ induced by the involution on $S^{\marked}$ interchanging the marked roots. 
       Then as rational maps,
       $$f = \delta^{-1} \circ \tau \circ \iota \circ \delta.$$
   \end{itemize}
\end{thm}

\begin{proof}
 The proof is carried out over $\C$ in \cite[Section 4]{MR3161305}. The same argument works without changes over any algebraically closed field for which the result of Theorem \ref{thm_spectral_transform} holds, again making the necessary change of replacing $ab$-coordinates with corner invariants.
\end{proof}

We now have the pieces of our main theorem.
\begin{proof}[Proof of Theorem \ref{thm_main_1}]
Theorem \ref{thm_moduli}, Theorem \ref{thm_spectral_transform}, Theorem \ref{thm_translation}.
\end{proof}

\new{We end this section with an application of Theorem \ref{thm_translation} to arithmetic complexity, partially confirming an empirical observation made in \cite[Section 5]{MR3282373}. Let $k = \bar{\Q}$, and let
$$ h_{\Tw}^{\Weil} : \Tw(\bar{\Q}) \to \R$$
be a (logarithmic) Weil height function on $\Tw$ relative to some chosen divisor; see \cite{MR1745599} for background. In \cite[Fig. 3]{MR3282373}, it is observed that the growth of $h_{\Tw}^{\Weil}$ appears to be polynomial. In fact, for sufficiently generic polygons, the height growth is linear.}

\new{
\begin{cor} \label{cor_height_growth}
Let $n \geq 4$, and let $v \in \Tw(\bar{\Q})$ be a twisted $n$-gon defined over $\bar{\Q}$ that is in the domain of $\delta$ and the image of $\delta^{-1}$. There is a constant $C = C(v) > 0$ depending on $v$, such that, for all $t \in \N$ such that $f^t(v)$ is defined, we have
$$h_{\Tw}^{\Weil} (f^t(v)) \leq C t.$$
\end{cor}
}
\begin{proof}
\new{
First, assume $n$ is odd. Let $A$ be the fiber of $\abfam \to S$ containing $\delta(v)$. According to Theorem \ref{thm_spectral_transform}, there exists a point $a \in A$ such that, for any twisted $n$-gon $w$ with $\delta(w) \in A$, we have
$$f(w) = \delta^{-1}(\delta(w) + a).$$
By functoriality of Weil heights, the pullback of $h_{\Tw}^{\Weil}$ to $A$ is a Weil height on $A$, which we denote $h_{A}^{\Weil}$. 
Thus, for all $t \in \N$ such that $f^t(v)$ is defined, we have
$$h_{\Tw}^{\Weil} (f^t(v)) = h_A^{\Weil}(\delta(v) + ta).$$
The growth of the right side is well-known to be at most linear in $t$. For completeness, we give the argument here. Since $A$ is an abelian variety defined over $\mathbb{\Q}$, it admits a canonical height $\hat{h}_A: A \to \R$; see \cite[Chapter B.5]{MR1745599} for all the properties we use. There exists a constant $C_0$ such that, for all $a_0 \in A$, we have
$$|\hat{h}_A(a_0)  - h_A^{\Weil}(a_0) | \leq C_0.$$
Further, for any $m \in \N$,
$$\hat{h}_A (m a) = m \hat{h}_A (m).$$
Finally, the parallelogram law implies that there is a constant $C_1$ independent of $w$ such that, for all $a_0 \in A$, we have
$$h_A^{\Weil}(\delta(v) + a_0) \leq 2 h_A^{\Weil}(a_0) + C_1. $$
Thus
\begin{align*}
h_A^{\Weil}(\delta(v) + ta) & \leq 2 h_A^{\Weil}(ta) + C_1 \\
& \leq 2 (\hat{h}_A (ta) + C_0) + C_1 \\
& \leq 2 t (\hat{h}_A (a) + C_2)
\end{align*}
for some sufficiently large constant $C_2$, completing the proof for $n$ odd.}

\new{Now let $n$ be even. The proof above shows that the height growth of the second iterate $f^2$ is linear. Since the growth is linear for the orbits of $f^2$ starting from both $v$ and $f(v)$, the orbit of $f$ starting at $v$ also has linear growth; just take the larger of the two constants.
}
\end{proof}

\section{A Collapse Conjecture for the Pentagram Map over Finite Fields} \label{sect_dynamical_domains}

Theorem \ref{thm_main_1} shows that the same algebro-geometric structure underlies the pentagram map over various fields. However, the implications for the dynamics over $\C$ and $\bar{\F}_p$ could not be more different. In this section, we sketch an argument that a randomly chosen twisted polygon in $\bar{\F}_p$, upon iteration, is likely to enter the degeneracy locus of the pentagram map. Geometrically, this corresponds to some iterate of the map being a degenerate polygon. This section may be read independently of the proof of Theorem \ref{thm_main_1}.

\begin{defn} \label{def_rat_orbit}
Given a rational map $\phi \colon V \dashrightarrow V$, let $\Ind_\phi \subset V$ denote the indeterminacy locus of~$\phi$. Given $N \in \N$, if $v \in \Ind_{\phi^N}$, we say \emph{the $N$-th iterate of $\phi$ at $v$ is undefined} and write formally
$$\phi^N(v) = *.$$
Otherwise, the point $\phi^N(x)$ is the \emph{$N$-th iterate of $x$}.

The \emph{orbit} of a point $v \in V$ is the set of iterates of $v$, valued in $V \cup \{*\}$. We also use the term \emph{orbit} to refer to the sequence of iterates of $v$, rather than the set they form.

A \emph{preperiodic point} is a point $v$ such that, for some values $M \neq N$ in $\N$, the iterates $\phi^M(v)$ and $\phi^N(v)$ are defined and $\phi^M(v) = \phi^N(v)$. A \emph{periodic point} of \emph{period $N$} is a point $v$ such that $\phi^N(v) = v$.
\end{defn}

Note that, if we were to use the inductive definition of orbit
$$\phi^0(v) = v, \quad \phi^{N+1}(v) = \phi(\phi^N(v)),$$
then each orbit would terminate after hitting the indeterminacy locus. Definition \ref{def_rat_orbit} makes it possible, in certain cases, to talk about the behavior of an orbit even \emph{after} meeting the indeterminacy locus.

\begin{example} \label{ex_crem}
Consider the Cremona involution on $\PP^2$. In homogeneous coordinates, the map is defined by
$$\phi: \PP^2 \dashrightarrow \PP^2,$$
$$[X: Y : Z] \mapsto [YZ : XZ: XY].$$
Since $\phi^2$ is the identity, every point is periodic with period 1 or 2. The point $v = [1:0:0]$ is in $\Ind_\phi$,
and
$$\shO_\phi(v) = v, \; *, \; v, \; *, \; \ldots$$
Let $w \in \PP^2$ be any point such that $X = 0$ and $Y, Z \neq 0$. Then
$$\shO_\phi(w) = w, \; v, \; w, \; v, \; \ldots$$
Thus different choices of $w$ give orbits which pass through the same point $v$ but then ``remember'' information from earlier in the orbit.
\end{example}

Our definition of orbit, Definition \ref{def_rat_orbit}, is unusual at first glance, but it is well-suited for the algebraic dynamics of rational maps, as we now explain.

In real and complex dynamics, it is common to restrict the domain so that a rational map $\phi \colon V \dashrightarrow V$ becomes a function, set-theoretically. This means throwing away all points with an iterate in $\Ind_\phi$. We call the remaining subset the \emph{dynamical domain}, denoted $\DynDom(\phi)$. Formally, writing $\phi^{-N}$ for the $N$-fold inverse image,
$$\DynDom(\phi) := V \smallsetminus \bigcup_{N = 0}^\infty \phi^{-N}(\Ind_\phi).$$

Over $\R$ and $\C$, if $\phi$ is dominant, then since $\Ind_\phi$ has measure~$0$, we end up deleting a set of measure~$0$ 
to obtain $\DynDom(\phi)$. This leaves a nice measure space on which to study the generic dynamics of the system.

\new{Over countable algebraically closed fields, such as $\bar{\Q}$ and $\bar{\F}_p$,} there is a potential problem in setting up the dynamical domain: a countable union of proper subvarieties can contain all the points of the domain. However, it turns out that the dynamical domain is never empty \new{over $\bar{\F}_p$}. By work of Hrushovski, $\DynDom(\phi)$ is Zariski dense in $V$; see \cite[Corollary 2]{MR2784670}. Nevertheless, on the level of points, the dynamical domain could still be very small. To measure the size of the dynamical domain, one can study the fraction of orbits which degenerate while going up a tower of finite fields. 

In the survey of arithmetic dynamics \cite{MR4007163}, Conjecture 18.10b states that the dynamical domain is large in the sense that, for any rational map $\phi: \PP^n \dashrightarrow \PP^n$ over $\F_q$,
$$\lim_{r \to \infty} \frac{\# \DynDom(\phi) \cap \PP^n(\F_{q^r} )}{\# \PP^n(\F_{q^r})} = 1.$$
But computer experiments and Theorem \ref{thm_main_1} suggest that the pentagram map is a counterexample.
\begin{conj} \label{conj_small_domain}
Let $p$ be a prime, and let $f$ denote the pentagram map on the moduli space of twisted $n$-gons over $\bar{\F}_p$. Then
$$\lim_{r \to \infty} \frac{\# \DynDom(f) \cap \Tw(\F_{p^r})}{\# \Tw(\F_{p^r})} = 0.$$
\end{conj}
From the formulas for the pentagram map in Proposition \ref{prop_formulas}, homogenizing via the scaling invariance, one obtains an explicit map $\phi: \PP^{2n - 1} \dashrightarrow \PP^{2n - 1}$ with the same property.

If we are right about Conjecture \ref{conj_small_domain}, then the dynamical domain is badly behaved over finite fields, \new{similar results might hold for other discrete integrable systems}. On the other hand, our main theorem shows that a birational change of domain can almost completely remove the dynamical indeterminacy of the map. The situation here is reminiscent of the notions of good reduction versus potential good reduction in arithmetic dynamics. It would be interesting to see whether more general maps admit conjugacies that improve their dynamical domains.

\subsection{Plausibility argument}
We now explain how Theorem \ref{thm_main_1} provides heuristic evidence for Conjecture \ref{conj_small_domain}.

For simplicity, let $n$ be odd. Let $S$ be the base of the family $\abfam$ of abelian varieties in Theorem \ref{thm_main_1}. Fix $p$ and consider various powers $q = p^r$. Each $\F_q$-point of $S$ gives us an $f$-invariant subvariety of $\Tw$ defined over $\F_q$. Some of the $\F_q$-points of $\Tw$ belong to invariant subvarieties defined over smaller finite fields than $\F_q$, but we claim these are relatively rare. Let $q = p^r$. For each proper subfield $\F_{p^\ell}$ of $\F_q$, we see $O(p^{(n+1)\ell})$ members of $\abfam$ defined over $\F_{p^\ell}$, each containing $O(p^{(n-1)r})$ points over $\F_q$. But these make up only a small fraction of the $O(p^{2nr})$ points in $\Tw(\F_q)$. 

The pentagram map is undefined at a point if the construction of the image polygon results in a degenerate $n$-gon; this occurs if some set of points of the form $v_i, v_{i+1}, v_{i+2}$ or $v_i, v_{i+2}, v_{i+4}$ are collinear.

Let $A$ be a member of $\abfam$ defined over $\F_q$. We claim almost all the $\F_q$-points eventually hit $\Ind_f \cap A$. To see this, observe that the degeneracy locus has codimension 1, so it should cut a codimension 1 subset out of each invariant fiber. So $\Ind_f \cap A$ contains $O(1/q)$ of the $\F_q$-points of $A$. The group of points of a random Jacobian over $\F_q$ is usually close to being cyclic, so the order of a random translation on a random Jacobian is close to the order of the whole group. Thus, tracking the motion on the Jacobian, $A(\F_q)$ should be roughly one orbit under $f$, so hitting a degeneracy point is highly likely.

To formalize the above argument, some substantial work would be needed to understand the intersections $\Ind_f \cap A$, to understand the randomness of the Jacobians and the translations, and to make sure that the degeneracy locus is not collapsed by~$\delta$.

\bibliographystyle{plain}
\bibliography{bib}
\end{document}